\documentclass[12pt,a4paper,oneside]{amsart}

\usepackage{amsmath}
\usepackage{amsthm}
\usepackage{amsfonts}
\usepackage{amssymb}

\usepackage[T1]{fontenc}
\usepackage{textcomp}

\usepackage[warnunknown]{ucs}
\usepackage[utf8x]{inputenc}

\usepackage[textwidth=12cm,textheight=20cm,centering]{geometry}

\usepackage{booktabs}

\makeatletter
\renewenvironment{abstract}{%
  \ifx\maketitle\relax
    \ClassWarning{\@classname}{Abstract should precede
      \protect\maketitle\space in AMS document classes; reported}%
  \fi
  \global\setbox\abstractbox=\vtop \bgroup
    \normalfont
    \list{}{\labelwidth\z@
      \leftmargin0pc \rightmargin\leftmargin
      \listparindent\normalparindent \itemindent\z@
      \parsep\z@ \@plus\p@
      
    }%
    \item[\hskip\labelsep\bfseries\abstractname:]%
}{%
  \endlist\egroup
  \ifx\@setabstract\relax \@setabstracta \fi
}
\def\@setabstracta{%
  \ifvoid\abstractbox
  \else
    \skip@16\p@ \advance\skip@-\lastskip
    \advance\skip@-\baselineskip \vskip\skip@
    \box\abstractbox
    \prevdepth\z@ 
  \fi
}
\def\@setsubjclass{%
\noindent\textbf{AMS Subj.~Classification:}\enspace\@subjclass}
\renewcommand{\keywordsname}{Key words}
\def\@setkeywords{%
  \par\noindent{\bfseries \keywordsname:}\enspace \@keywords\@addpunct.}
\def\maketitle{\par
  \@topnum\z@ 
  \@setcopyright
  \thispagestyle{empty}
  \uppercasenonmath\shorttitle
  \ifx\@empty\shortauthors \let\shortauthors\shorttitle
  \else \andify\shortauthors
  \fi
  \@maketitle@hook
  \begingroup
  \@maketitle
  \toks@\@xp{\shortauthors}\@temptokena\@xp{\shorttitle}%
  \toks4{\def\\{ \ignorespaces}}
  \edef\@tempa{%
    \@nx\markboth{\the\toks4
      \@nx\MakeUppercase{\the\toks@}}{\the\@temptokena}}%
  \@tempa
  \endgroup
  \c@footnote\z@
  \@cleartopmattertags
}
\def\@setauthors{%
  \begingroup
  \def\thanks{\protect\thanks@warning}%
  \trivlist
  \centering\footnotesize \@topsep24\p@\relax
  \advance\@topsep by -\baselineskip
  \item\relax
  \author@andify\authors
  \def\\{\protect\linebreak}%
  \MakeUppercase{\authors}%
  \ifx\@empty\contribs
  \else
    ,\penalty-3 \space \@setcontribs
    \@closetoccontribs
  \fi
  \endtrivlist
  \endgroup
}
\def\emailaddrname{{e-mail}}
\def\@setaddresses{\par
  \nobreak \begingroup\centering\footnotesize
  \def\author##1{\nobreak\addvspace\bigskipamount}%
  \interlinepenalty\@M
  \def\address##1##2{\begingroup
    \par\vspace{3pt}
    \@ifnotempty{##1}{(\ignorespaces##1\unskip) }%
    {\ignorespaces##2}\par\endgroup}%
  \def\curraddr##1##2{\begingroup
    \@ifnotempty{##2}{\nobreak\indent\curraddrname
      \@ifnotempty{##1}{, \ignorespaces##1\unskip}\/:\space
      ##2\par}\endgroup}%
  \def\email##1##2{\begingroup
    \@ifnotempty{##2}{\nobreak\indent\emailaddrname
      \@ifnotempty{##1}{, \ignorespaces##1\unskip}\/:\space
      \ttfamily##2\par}\endgroup}%
  \def\urladdr##1##2{\begingroup
    \def~{\char`\~}%
    \@ifnotempty{##2}{\nobreak\indent\urladdrname
      \@ifnotempty{##1}{, \ignorespaces##1\unskip}\/:\space
      \ttfamily##2\par}\endgroup}%
  \addresses
  \endgroup
}
\def\@maketitle{%
  \normalfont\normalsize
  \@adminfootnotes
  \@mkboth{\@nx\shortauthors}{\@nx\shorttitle}%
  \vspace*{3.5cm}
  \@settitle
  \ifx\@empty\authors \else \@setauthors \fi
  \ifx\@empty\addresses \else\@setaddresses\fi
  \smallskip
  \ifx\@empty\@dedicatory
  \else
    \baselineskip18\p@
    \vtop{\centering{\footnotesize\itshape\@dedicatory\@@par}%
      \global\dimen@i\prevdepth}\prevdepth\dimen@i
  \fi
  \@setabstract
  \@setsubjclass
  \@setkeywords
  \normalsize
  \if@titlepage
    \newpage
  \else
    \dimen@10\p@ \advance\dimen@-\baselineskip
    \vskip\dimen@\relax
  \fi
} 
\def\@adminfootnotes{%
  \let\@makefnmark\relax  \let\@thefnmark\relax
  \ifx\@empty\@date\else \@footnotetext{\@setdate}\fi
  \ifx\@empty\thankses\else \@footnotetext{%
    \def\par{\let\par\@par}\@setthanks}%
  \fi
}
\def\section{\@startsection{section}{1}%
  \z@{.7\linespacing\@plus\linespacing}{.5\linespacing}%
  {\normalfont\bfseries\centering}}
\def\@secnumfont{\bfseries}
\def\enddoc@text{}
\makeatother
\newtheoremstyle{theorem}
{10pt} 
{10pt} 
{\slshape} 
{\parindent} 
{\bfseries} 
{. } 
{ } 
{} 
\newtheoremstyle{remark}
{10pt} 
{10pt} 
{\upshape} 
{\parindent} 
{\bfseries} 
{. } 
{ } 
{} 
\theoremstyle{theorem}
\newtheorem{theorem}{Theorem}[section]

\newtheorem{lemma}[theorem]{Lemma}
\newtheorem{proposition}[theorem]{Proposition}
\theoremstyle{remark}
\newtheorem{remark}{Remark}[section]
\newtheorem{example}{Example}[section]
\numberwithin{equation}{section}

\DeclareMathOperator{\diag}{diag}

\begin{document}
\title[Sturm-Liouville problem and Polar Representation]{The Sturm-Liouville problem and the Polar Representation Theorem}
\author{Jorge Rezende}
\address{Grupo de F\'{i}sica-Matem\'{a}tica da Universidade de Lisboa\\
Av. Prof. Gama Pinto 2, 1649-003 Lisboa, PORTUGAL\\
and\\
Departamento de Matem\'{a}tica,\\
Faculdade de Ci\^{e}ncias da Universidade de Lisboa}
\email{rezende@cii.fc.ul.pt}
\dedicatory{Dedicated to the memory of Professor Ruy Lu\'{i}s Gomes}

\begin{abstract}
The polar representation theorem for the $n$-dimensional time-dependent
linear Hamiltonian system 
\begin{equation*}
\dot{Q}=BQ+CP\text{, \ }\dot{P}=-AQ-B^{*}P\text{,}
\end{equation*}
with continuous coefficients, states that, given two isotropic solutions $%
\left( Q_1,P_1\right) $ and $\left( Q_2,P_2\right) $, with the identity
matrix as Wronskian, the formula 
\begin{equation*}
Q_2=r\cos \varphi \text{, \ }Q_1=r\sin \varphi \text{,}
\end{equation*}
holds, where $r$ and $\varphi $ are continuous matrices, $\det r\neq 0$ and $%
\varphi $ is symmetric.

In this article we use the monotonicity properties of the matrix $\varphi $
eigenvalues in order to obtain results on the Sturm-Liouville problem.
\end{abstract}
\maketitle

\section{Introduction}

\noindent Let $n=1,2,\ldots $. In this article, $\left( .,.\right) $ denotes the
natural inner product in $\mathbb{R}^{n}$. For $x\in \mathbb{R}^{n}$ one writes $%
x^{2}=\left( x,x\right) $, $\left| x\right| =\left( x,x\right) ^{\frac{1}{2}%
} $. If $M$ is a real matrix, we shall denote $M^{\ast }$ its transpose. $%
M_{jk}$ denotes the matrix entry located in row $j$ and column $k $. $I_{n}$
is the identity $n\times n$ matrix. $M_{jk}$ can be a matrix. For example, $%
M $ can have the four blocks $M_{11}$, $M_{12}$, $M_{21}$, $M_{22}$. In a
case like this one, if $M_{12}=M_{21}=0$, we write $M=\diag\left(
M_{11},M_{22}\right) $.

\subsection{The symplectic group and the polar representation theorem}
\ \\
\noindent Consider the time-dependent linear Hamiltonian system 
\begin{equation}
\dot{Q}=BQ+CP\text{, }\dot{P}=-AQ-B^{\ast }P\text{,}  \label{J1}
\end{equation}
where $A$, $B$ and $C$ are time-dependent $n\times n$ matrices. $A$ and $C$
are symmetric. The dot means time derivative, the derivative with respect to 
$\tau $. The time variable $\tau $ belongs to an interval. Without loss of
generality we shall assume that this interval is $\left[ 0,T\right[ $,
$T>0$. $T$ can be $\infty $. In the following $t$, $0<t<T$, is also a time
variable and $\tau \in \left[ 0,t\right] $.

If $( Q_1,P_1) $ and $( Q_2,P_2) $ are solutions of 
\eqref{J1} one denotes the Wronskian (which is constant) by
\begin{equation*}
W( Q_1,P_1;Q_2,P_2) \equiv W=P_1^{*}Q_2-Q_1^{*}P_2\text{.}
\end{equation*}

A solution $( Q,P) $ of \eqref{J1} is called isotropic if 
$W( Q,P;Q,P) =0$. From now on $( Q_1,P_1) $ and $(
Q_2,P_2) $ will denote two isotropic solutions of \eqref{J1} such that 
$W( Q_1,P_1;Q_2,P_2) =I_n$. This means that 
\begin{equation*}
P_1^{*}Q_2-Q_1^{*}P_2=I_n\text{ , }P_1^{*}Q_1=Q_1^{*}P_1\text{ , }%
P_2^{*}Q_2=Q_2^{*}P_2\text{.}
\end{equation*}

These relations express precisely that, for each $\tau \in \left[ 0,T\right[ 
$ the $2n\times 2n$ matrix 
\begin{equation}
\Phi =
\begin{bmatrix}
Q_{2} & Q_{1} \\ 
P_{2} & P_{1}
\end{bmatrix}
\label{n}
\end{equation}
is symplectic. Its left inverse and, therefore, its inverse, is given by 
\begin{equation*}
\Phi ^{-1}=
\begin{bmatrix}
P_{1}^{\ast } & -Q_{1}^{\ast } \\ 
-P_{2}^{\ast } & Q_{2}^{\ast }
\end{bmatrix}
\text{.}
\end{equation*}

As it is well-known the $2n\times 2n$ symplectic matrices form a group, the
symplectic group.

Then, one has 
\begin{equation*}
P_1Q_2^{*}-P_2Q_1^{*}=I_n,\quad Q_1Q_2^{*}=Q_2Q_1^{*},\quad
P_1P_2^{*}=P_2P_1^{*},
\end{equation*}
and, therefore, 
\begin{equation*}
Q_2^{*}P_1-P_2^{*}Q_1=I_n,\quad Q_2P_1^{*}-Q_1P_2^{*}=I_n\text{,}
\end{equation*}
and the following matrices, whenever they make sense, are symmetric 
\begin{equation*}
\begin{aligned}
&P_2Q_2^{-1},\quad Q_1P_1^{-1},\quad Q_2P_2^{-1},\quad P_1Q_1^{-1}\text{,} \\
&Q_2^{-1}Q_1,\quad P_2^{-1}P_1,\quad Q_1^{-1}Q_2,\quad P_1^{-1}P_2\text{.}
\end{aligned}
\end{equation*}

Denote by $J$, $S$ and $M$, the following $2n\times 2n$ matrices 
\begin{equation*}
J=
\begin{bmatrix}
0 & -I_{n} \\ 
I_{n} & 0
\end{bmatrix},\quad
S=
\begin{bmatrix}
A & B^{\ast } \\ 
B & C
\end{bmatrix}
\text{,}
\end{equation*}
and $M=-JS$. $J$ is symplectic and $S$ is symmetric.

One says that the $2n\times 2n$ matrix $L$ is antisymplectic if $LJL^{\ast
}=-J$. Notice that the product of two antisymplectic matrices is symplectic,
and that the product of an antisymplectic matrix by a symplectic one is
antisymplectic. We shall use this definition later.

Notice that if $n=1$ and $L$ is a $2\times 2$ matrix, then one has 
$LJL^{\ast }=\left( \det L\right) J$.

Equation \eqref{J1} can then be written 
\begin{equation*}
\dot{\Phi}=M\Phi \text{.}
\end{equation*}

Notice that, if $\Phi $ is symplectic, $\Phi ^{*}$ is symplectic, and 
\begin{equation*}
\Phi ^{-1}=-J\Phi ^{*}J\text{, \ }\Phi ^{*}J\Phi =J\text{, \ }\Phi J\Phi
^{*}=J\text{.}
\end{equation*}

When we have a $C^1$ function $\tau \longmapsto \Phi \left( \tau \right) $, 
$\dot{\Phi}J\Phi ^{*}+\Phi J\dot{\Phi}^{*}=0$. Hence, $\dot{\Phi}J\Phi ^{*}$
is symmetric and one can recover $M$: 
\begin{equation*}
M=\dot{\Phi}\Phi ^{-1}=-\dot{\Phi}J\Phi ^{*}J\text{.}
\end{equation*}

This means that from $\Phi $ one can obtain $A$, $B$, and $C$: 
\begin{equation*}
\begin{aligned}
A&=\dot{P}_{1}P_{2}^{\ast }-\dot{P}_{2}P_{1}^{\ast }\text{,\quad}C=\dot{Q}%
_{1}Q_{2}^{\ast }-\dot{Q}_{2}Q_{1}^{\ast }\text{,}\\
B&=-\dot{Q}_{1}P_{2}^{\ast }+\dot{Q}_{2}P_{1}^{\ast }=Q_{1}\dot{P}_{2}^{\ast
}-Q_{2}\dot{P}_{1}^{\ast }\text{.}
\end{aligned}
\end{equation*}

The proof of the following theorem on a polar representation can be found in 
\cite{rezende1}. See also \cite{rezende2}, \cite{rezende4}.

\begin{theorem}
\label{T1}Assume that $C\left( \tau \right) $ is always $>0$ (or always $<0$%
) and of class $C^1$. Consider two isotropic solutions of \eqref{J1}, $%
\left( Q_1,P_1\right) $ and $\left( Q_2,P_2\right) $, such that $W=I_n$.
Then, there are $C^1$ matrix-valued functions $r\left( \tau \right) $, $%
\varphi \left( \tau \right) $, for $\tau \in \left[ 0,T\right[ $, such that:
a) $\det r\left( \tau \right) \neq 0$ and $\varphi \left( \tau \right) $ is
symmetric for every $\tau $; b) the eigenvalues of $\varphi $ are $C^1$
functions of $\tau $, with strictly positive (negative) derivatives; c) one
has 
\begin{equation*}
Q_2\left( \tau \right) =r\left( \tau \right) \cos \varphi \left( \tau
\right) \text{ \ and \ }Q_1\left( \tau \right) =r\left( \tau \right) \sin
\varphi \left( \tau \right) \text{.}
\end{equation*}
\end{theorem}

Notice that $\varphi $ is not unique and that
\begin{equation}
\frac{d}{d\tau }Q_{2}^{-1}Q_{1}=Q_{2}^{-1}CQ_{2}^{\ast -1}\text{,}  \label{o}
\end{equation}
whenever $\det Q_{2}\left( \tau \right) \neq 0$ (see \cite{rezende1}).

\begin{example}
Consider $n=1$, $B=0$, $A=C=1$. Let $k_1,k_2\in \mathbb{R}$. For $k_2>0$, let 
\begin{equation*}
Q_2( \tau ) =k_2^{-1/2}\cos \tau \text{, \ }Q_1( \tau
) =k_2^{-1/2}( k_1\cos \tau +k_2\sin \tau ) \text{.}
\end{equation*}
Then there exists an increasing continuous function of $\tau $,$\ \xi (
k_1,k_2,\tau ) \linebreak[0]\equiv \xi ( \tau ) $, $\tau \in \mathbb{R}$,
such that 
\begin{equation*}
Q_2( \tau ) =r( \tau ) \cos \xi ( \tau ) 
\text{, \ }Q_1( \tau ) =r( \tau ) \sin \xi (
\tau ) \text{,}
\end{equation*}
where $r( \tau ) =
k_2^{-1/2}\sqrt{\cos ^2\tau +( k_1\cos
\tau +k_2\sin \tau) ^2}$. The function $\xi $ is not unique in the
sense that two such functions differ by $2k\pi $, $k\in \mathbb{Z}$. For $\tau
\neq \frac \pi 2+k\pi $, one has 
\begin{equation}
k_1+k_2\tan \tau =\tan \xi ( \tau ) .  \label{l}
\end{equation}

This formula shows that $\lim_{\tau \rightarrow \pm \infty }\xi ( \tau
) =\pm \infty $.

For $k_2<0$, one defines, obviously, 
\begin{equation*}
\xi ( k_1,k_2,\tau ) =-\xi ( -k_1,-k_2,\tau ) \text{.}
\end{equation*}

When $k_2=0$, $\xi $ is a constant function. For every $k_2\in \mathbb{R}$,
formula \eqref{l} remains valid.

One can fix $\xi $ by imposing $-\frac \pi 2<\xi ( 0) <\frac \pi
2 $, as we shall do from now on.

For $k_2>0$, one has $\xi \left( \frac \pi 2+k\pi \right) =\frac \pi 2+k\pi $%
, and for $k_2<0$, one has $\xi \left( \frac \pi 2+k\pi \right) =-\frac \pi
2-k\pi $, for every $k\in \mathbb{Z}$.

If $S$ is a symmetric $n\times n$\ matrix, and $\Omega $ is an orthogonal
matrix that diagonalizes $S$, $S=\Omega \diag( s_1,s_2,\ldots
,s_n) \Omega ^{*}$, we denote 
\begin{equation*}
\xi ( k_1,k_2,S) \equiv \xi ( S) =\Omega \diag%
( \xi ( s_1) ,\xi ( s_2) ,\ldots ,\xi (
s_n) ) \Omega ^{*}\text{.}
\end{equation*}

Define now 
\begin{equation}
\zeta ( \tau ) \equiv \zeta ( k_1,k_2,\tau ) =-\xi
( k_1,k_2,\tau ) +\frac \pi 2\text{.}  \label{m}
\end{equation}
Then $0<\zeta ( 0) <\pi $, and 
\begin{equation*}
( k_1+k_2\tan \tau ) ^{-1}=\tan \zeta ( \tau ) ,
\end{equation*}
for every $\tau $ such that $k_1+k_2\tan \tau \neq 0$.

For $k_2>0$, one has $\zeta \left( \frac \pi 2+k\pi \right) =-k\pi $, and
for $k_2<0$, one has $\zeta \left( \frac \pi 2+k\pi \right) =
(k+1) \pi $, for every $k\in \mathbb{Z}$. The function $\zeta $ is
increasing for $k_2<0$, decreasing for $k_2>0$ and constant for $k_2=0$.

If $S$ is a symmetric $n\times n$\ matrix, one can define $\zeta (
k_1,k_2,S) $ as we did before for $\xi $.

We shall need these functions later.
\end{example}

Theorem \ref{T1}\ can be extended in the following way:

\begin{theorem}
\label{T3}Assume that $C\left( \tau \right) $ is of class $C^1$. Consider
two isotropic solutions of \eqref{J1}, $\left( Q_1,P_1\right) $ and $\left(
Q_2,P_2\right) $, such that $W=I_n$. Then, there are $C^1$ matrix-valued
functions $r\left( \tau \right) $, $\varphi \left( \tau \right) $, for $\tau
\in \left[ 0,t\right] $, such that: a) $\det r\left( \tau \right) \neq 0$
and $\varphi \left( \tau \right) $ is symmetric for every $\tau $; b) the
eigenvalues of $\varphi $ are $C^1$ functions of $\tau $; c) one has 
\begin{equation*}
Q_2\left( \tau \right) =r\left( \tau \right) \cos \varphi \left( \tau
\right) \text{ \ and \ }Q_1\left( \tau \right) =r\left( \tau \right) \sin
\varphi \left( \tau \right) \text{.}
\end{equation*}
\end{theorem}

\begin{proof}
Let us first notice that $Q_2Q_2^{*}+Q_1Q_1^{*}>0$. This is proved noticing
that, as $P_1Q_2^{*}-P_2Q_1^{*}=I_n$, one has $\left(
P_1^{*}x,Q_2^{*}x\right) -\left( P_2^{*}x,Q_1^{*}x\right) =\left| x\right|
^2 $, which implies that $\ker Q_1^{*}\cap \ker Q_2^{*}=\left\{ 0\right\} $.
Hence, $\left( Q_2^{*}x,Q_2^{*}x\right) +\left( Q_1^{*}x,Q_1^{*}x\right) 
>0$, for every $x\neq 0$.

Define now 
\begin{equation*}
\Phi =
\begin{bmatrix}
Q_2 & Q_1 \\ 
P_2 & P_1
\end{bmatrix}
\text{, \ }\Psi =
\begin{bmatrix}
\cos \left( k\tau \right) I_n & \sin \left( k\tau \right) I_n \\ 
-\sin \left( k\tau \right) I_n & \cos \left( k\tau \right) I_n
\end{bmatrix}
\text{,}
\end{equation*}
$M$ as before, $\Phi _1=\Phi \Psi $ and $M_1=\dot{\Phi}_1\Phi _1^{-1}$. The
constant $k$ is $>0$. Then, one has 
\begin{equation*}
M_1=M+\Phi \dot{\Psi}\Psi ^{-1}\Phi ^{-1}\text{.}
\end{equation*}

Let the $n\times n$ matrices, that are associated with $M_1$, be $A_1$, $B_1$
and $C_1$. Then 
\begin{equation*}
C_1=C+k\left( Q_2Q_2^{*}+Q_1Q_1^{*}\right) \text{.}
\end{equation*}

Hence, as $Q_2Q_2^{*}+Q_1Q_1^{*}>0$, for $k$ large enough, we have that 
$C_1\left( \tau \right) >0$, for every $\tau \in \left[ 0,t\right] $. We can
then apply Theorem \ref{T1}. There are $C^1$ matrix-valued functions $%
r_1\left( \tau \right) $, $\varphi _1\left( \tau \right) $, for $\tau \in
\left[ 0,t\right] $, such that 
\begin{align*}
\cos \left( k\tau \right) Q_2\left( \tau \right) -\sin \left( k\tau \right)
Q_1\left( \tau \right) &=r_1\left( \tau \right) \cos \varphi _1\left( \tau
\right) \\
\sin \left( k\tau \right) Q_2\left( \tau \right) +\cos \left( k\tau \right)
Q_1\left( \tau \right) &=r_1\left( \tau \right) \sin \varphi _1\left( \tau
\right) \text{.}
\end{align*}
From this, we have 
\begin{align*}
Q_2\left( \tau \right) &= r_1\left( \tau \right) \cos \left( \varphi
_1\left( \tau \right) -k\tau I_n\right) \\
Q_1\left( \tau \right) &= r_1\left( \tau \right) \sin \left( \varphi
_1\left( \tau \right) -k\tau I_n\right) \text{.}
\end{align*}
\end{proof}

The generic differential equations for $r$ and $\varphi $ are easily derived
from equations (15), (17) and (18) in \cite{rezende1}.

Consider $\left( r_0,s\right) $, with $s$ symmetric, such that 
\begin{equation*}
\dot{r}_0=Br_0+Cr_0^{*-1}s\text{, \ }\dot{s}%
=sr_0^{-1}Cr_0^{*-1}s+r_0^{-1}Cr_0^{*-1}-r_0^{*}Ar_0\text{.}
\end{equation*}
Then $r$ is of the form $r=r_0\Omega $, where $\Omega $ is any orthogonal, 
$\Omega ^{-1}=\Omega ^{*}$, and time-dependent $C^1$ matrix. From this one
can derive a differential equation for $rr^{*}$.

The function $\varphi $ verifies the equations 
\begin{equation}
\frac{\cos \mathcal{C}_{\varphi }-I}{\mathcal{C}_{\varphi }}\dot{\varphi}%
=-\Omega ^{\ast }\dot{\Omega}\text{,\quad}\frac{\sin \mathcal{C}_{\varphi }%
}{\mathcal{C}_{\varphi }}\dot{\varphi}=r^{-1}Cr^{\ast -1}\text{,}  \label{k}
\end{equation}
where $\mathcal{C}_{\varphi }\dot{\varphi}=\left[ \varphi ,\dot{\varphi}%
\right] =\varphi \dot{\varphi}-\dot{\varphi}\varphi $, $\left( \mathcal{C}%
_{\varphi }\right) ^{2}\dot{\varphi}\equiv \mathcal{C}_{\varphi }^{2}\dot{%
\varphi}=\left[ \varphi ,\left[ \varphi ,\dot{\varphi}\right] \right] $, and
so on.

As in Theorem \ref{T1}, $\varphi $ is not unique. Notice that $r\left( \tau
\right) =r_{1}\left( \tau \right) $ and $\varphi \left( \tau \right)
=\varphi _{1}\left( \tau \right) -k\tau I_{n}$, with $k$ large enough and $%
\varphi _{1}$ such that its eigenvalues are $C^{1}$ functions of $\tau $,
with strictly positive derivatives.

\begin{remark}
If one considers $\Phi ^{*}$ instead of $\Phi $, then $Q_2$ is replaced by $%
Q_2^{*}$ and $Q_1$ is replaced by $P_1^{*}$. Then Theorem \ref{T3} gives
\begin{equation*}
Q_2^{*}\left( \tau \right) =r\left( \tau \right) \cos \varphi \left( \tau
\right) \text{ \ and \ }P_2^{*}\left( \tau \right) =r\left( \tau \right)
\sin \varphi \left( \tau \right) \text{,}
\end{equation*}
or 
\begin{equation*}
Q_2\left( \tau \right) =\cos \varphi \left( \tau \right) r^{*}\left( \tau
\right) \text{ \ and \ }P_2\left( \tau \right) =\sin \varphi \left( \tau
\right) r^{*}\left( \tau \right) \text{.}
\end{equation*}
In this case the matrix $\varphi \left( \tau \right) $ is a generalization
of the so-called Pr\"{u}fer angle \cite{CL}.
\end{remark}

Denote $\left( Q_c,P_c\right) $, $\left( Q_s,P_s\right) $ the (isotropic)
solutions of \eqref{J1} such that 
\begin{equation*}
Q_c\left( 0\right) =P_s\left( 0\right) =I_n\text{,\quad}Q_s\left( 0\right)
=P_c\left( 0\right) =0\text{.}
\end{equation*}
From now on we shall denote by $\Phi _0$ the symplectic matrix 
\begin{equation*}
\Phi _0=
\begin{bmatrix}
Q_c & Q_s \\ 
P_c & P_s
\end{bmatrix}
\text{.}
\end{equation*}
Then $\dot{\Phi}_0=M\Phi _0$ and $\Phi _0\left( 0\right) =I_{2n}$.

\subsection{The Sturm-Liouville problem}
\ \\
\noindent 
Let $t\in \left[ 0,T\right[ $ and $\lambda \in \left] l_{-1},l_{1}\right[
\subset \mathbb{R}$. The interval $\left] l_{-1},l_{1}\right[ $ can be as
general as possible. In this article, $t$ is the ''time'' variable and $%
\lambda $ is the ''eigenvalue'' variable.

Consider $A_0$, $B_0$ and $C_0$ time and eigenvalue dependent $n\times n$
matrices. As in \eqref{J1} $A_0$ and $C_0$ are symmetric. Define also $M_0$, 
$S_0$ and $\Phi _0$ (here, $\dot{\Phi}_0=M_0\Phi _0$) as before.

From now on we shall use the notations $A_0\equiv A_0\left( \tau \right)
\equiv A_0\left( \tau ,\lambda \right) $, and the same for the other
matrices.

Consider also $\alpha _{j}$, $\beta _{j}$, $\gamma _{j}$ and $\delta _{j}$, $%
j=0,1$, eight eigenvalue dependent $n\times n$ matrices, and the problem of
finding a $\lambda $ and a solution 
\begin{equation*}
\tau \longmapsto \left( q\left( \tau ,\lambda \right) ,p\left( \tau ,\lambda
\right) \right) \equiv \left( q\left( \tau \right) ,p\left( \tau \right)
\right) \equiv \left( q,p\right) \text{,}
\end{equation*}
$\left( q,p\right) \in \mathbb{R}^{n}\times \mathbb{R}^{n}$, for $\tau \in \left[
0,t\right] ,\lambda \in \left] l_{-1},l_{1}\right[ $, of the system 
\begin{equation*}
\dot{q}=B_{0}q+C_{0}p\text{,\quad}\dot{p}=-A_{0}q-B_{0}^{\ast }p\text{,}
\end{equation*}
with the ''boundary'' conditions 
\begin{equation*}
\begin{bmatrix}
\beta _{0} & \delta _{0} \\ 
\beta _{1} & \delta _{1}
\end{bmatrix}
\begin{bmatrix}
-q\left( 0\right) \\ 
q\left( t\right)
\end{bmatrix}
+
\begin{bmatrix}
-\alpha _{0} & \gamma _{0} \\ 
-\alpha _{1} & \gamma _{1}
\end{bmatrix}
\begin{bmatrix}
p\left( 0\right) \\ 
p\left( t\right)
\end{bmatrix}
=0\text{,}
\end{equation*}
or, equivalently, 
\begin{equation*}
\begin{bmatrix}
\beta _{0} & \alpha _{0} \\ 
\beta _{1} & \alpha _{1}
\end{bmatrix}
\begin{bmatrix}
q\left( 0\right) \\ 
p\left( 0\right)
\end{bmatrix}
-
\begin{bmatrix}
\delta _{0} & \gamma _{0} \\ 
\delta _{1} & \gamma _{1}
\end{bmatrix}
\begin{bmatrix}
q\left( t\right) \\ 
p\left( t\right)
\end{bmatrix}
=0\text{.}
\end{equation*}

Denote 
\begin{equation*}
\mathcal{S}_{q}=
\begin{bmatrix}
\beta _{0} & \delta _{0} \\ 
\beta _{1} & \delta _{1}
\end{bmatrix}
\text{,\quad}\mathcal{S}_{p}=
\begin{bmatrix}
-\alpha _{0} & \gamma _{0} \\ 
-\alpha _{1} & \gamma _{1}
\end{bmatrix}
\text{.}
\end{equation*}

In order to preserve the self-adjointness of the problem, one has to have
self-adjoint boundary conditions $\mathcal{S}_{q}\mathcal{S}_{p}^{\ast }=%
\mathcal{S}_{p}\mathcal{S}_{q}^{\ast }$ \cite{kratz}. This means that 
\begin{align*}
\alpha _{0}\beta _{0}^{\ast }+\delta _{0}\gamma _{0}^{\ast } &=\beta
_{0}\alpha _{0}^{\ast }+\gamma _{0}\delta _{0}^{\ast }, \\
\alpha _{1}\beta _{1}^{\ast }+\delta _{1}\gamma _{1}^{\ast } &=\beta
_{1}\alpha _{1}^{\ast }+\gamma _{1}\delta _{1}^{\ast }, \\
\alpha _{0}\beta _{1}^{\ast }+\delta _{0}\gamma _{1}^{\ast } &=\beta
_{0}\alpha _{1}^{\ast }+\gamma _{0}\delta _{1}^{\ast }.
\end{align*}

\begin{remark}
Consider $F$ a eigenvalue dependent symplectic matrix. If $\Phi $ is a
symplectic solution of $\dot{\Phi}=M_0\Phi $, then all previous formulas
involving $\Phi $, $M_0$, $\mathcal{S}_q$ and $\mathcal{S}_p$ remain valid
if we replace $\Phi $ by $F^{-1}\Phi $, $M_0$ by $F^{-1}M_0F$, $\mathcal{S}%
_q $ by $\mathcal{S}_q\diag( F_{11},F_{11}) +\mathcal{S}_p%
\diag( -F_{21},F_{21}) $, and $\mathcal{S}_p$ by $\mathcal{S%
}_q\diag( -F_{12},F_{12}) +\mathcal{S}_p\diag(
F_{22},F_{22}) $.
\end{remark}

As 
\begin{equation*}
\begin{bmatrix}
q\left( \tau \right) \\ 
p\left( \tau \right)
\end{bmatrix}
=\Phi _0\left( \tau \right)
\begin{bmatrix}
q\left( 0\right) \\ 
p\left( 0\right)
\end{bmatrix}
\end{equation*}
one obtains 
\begin{equation*}
\left(
\begin{bmatrix}
\beta _0 & \alpha _0 \\ 
\beta _1 & \alpha _1
\end{bmatrix}
-
\begin{bmatrix}
\delta _0 & \gamma _0 \\ 
\delta _1 & \gamma _1
\end{bmatrix}
\Phi _0\left( t\right) \right)
\begin{bmatrix}
q\left( 0\right) \\ 
p\left( 0\right)
\end{bmatrix}
=0.
\end{equation*}

In order to have a non trivial solution, $\left( q\left( 0\right) ,p\left(
0\right) \right) \neq \left( 0,0\right) $, of this system we must have 
\begin{equation}
\det \left(
\begin{bmatrix}
\beta _0 & \alpha _0 \\ 
\beta _1 & \alpha _1
\end{bmatrix}
-
\begin{bmatrix}
\delta _0 & \gamma _0 \\ 
\delta _1 & \gamma _1
\end{bmatrix}
\Phi _0\left( t\right) \right) =0.  \label{a}
\end{equation}

We shall need now the following lemma.

\begin{lemma}
Consider $a$, $b$, $c$ and $d$, $n\times n$ real matrices, such that $%
ab^{*}=ba^{*}$ and $cd^{*}=dc^{*}$. Let 
\begin{equation*}
N=
\begin{bmatrix}
a & b \\ 
c & d
\end{bmatrix}
\text{.}
\end{equation*}
Then $\det N=0$ if and only if $\det \left( ad^{*}-bc^{*}\right) =0$.
\end{lemma}

\begin{proof}
From 
\begin{equation*}
NJN^{*}J=\diag\left( -ad^{*}+bc^{*},-da^{*}+cb^{*}\right) \text{,}
\end{equation*}
one has $\left( \det N\right) ^2=\left( \det \left( ad^{*}-bc^{*}\right)
\right) ^2$. The lemma follows now easily.
\end{proof}

In order to apply this lemma to equation \eqref{a} we need to assume that,
from now on, 
\begin{equation}
\beta _j\alpha _j^{*}+\delta _j\gamma _j^{*}-\beta _jQ_s^{*}\left( t\right)
\delta _j^{*}-\beta _jP_s^{*}\left( t\right) \gamma _j^{*}-\delta
_jQ_c\left( t\right) \alpha _j^{*}-\gamma _jP_c\left( t\right) \alpha _j^{*}%
\text{,}  \label{b}
\end{equation}
for $j=0,1$, is symmetric.

Condition \eqref{b} is equivalent to 
\begin{equation*}
\begin{bmatrix}
\delta _{j} & \gamma _{j}
\end{bmatrix}
\Phi _{0} 
\begin{bmatrix}
-\alpha _{j} & \beta _{j}
\end{bmatrix}^{\ast }
+\beta _{j}\alpha _{j}^{\ast }+\delta _{j}\gamma _{j}^{\ast }\text{,}
\end{equation*}
for $j=0,1$, is symmetric. This is true for every symplectic matrix $\Phi
_{0}$ if and only if it is true for every matrix $\Phi _{0}$, even if it is
not symplectic. Then one can easily prove the following proposition.

\begin{proposition}
\begin{equation*}
\begin{bmatrix}
\delta _j & \gamma _j
\end{bmatrix}
\Phi _0
\begin{bmatrix}
-\alpha _j & \beta _j
\end{bmatrix}^{*}
+\beta _j\alpha _j^{*}+\delta _j\gamma _j^{*}\text{,}
\end{equation*}
for $j=0,1$, is symmetric for every symplectic matrix $\Phi _0$, if and only
if $\beta _j\alpha _j^{*}+\delta _j\gamma _j^{*}$ is symmetric and $\beta
_jG\delta _j^{*}=0$, $\beta _jG\gamma _j^{*}=0$, $\delta _jG\alpha _j^{*}=0$%
, $\gamma _jG\alpha _j^{*}=0$, for $j=0,1$, and every antisymmetric matrix 
$G$.
\end{proposition}

With this assumption, equation \eqref{a} is equivalent to 
\begin{equation}
\det \left( ad^{\ast }-bc^{\ast }\right) =0\text{,}  \label{t}
\end{equation}
where 
\begin{align*}
a &=\beta _{0}-\delta _{0}Q_{c}\left( t\right) -\gamma _{0}P_{c}\left(
t\right) \\
d &=\alpha _{1}-\delta _{1}Q_{s}\left( t\right) -\gamma _{1}P_{s}\left(
t\right) \\
b &=\alpha _{0}-\delta _{0}Q_{s}\left( t\right) -\gamma _{0}P_{s}\left(
t\right) \\
c &=\beta _{1}-\delta _{1}Q_{c}\left( t\right) -\gamma _{1}P_{c}\left(
t\right).
\end{align*}

It is then natural to consider a symplectic matrix $\Phi$ defined by 
\begin{equation*}
\Phi =
\begin{bmatrix}
Q_{2} & Q_{1} \\ 
P_{2} & P_{1}
\end{bmatrix}
\text{,}
\end{equation*}
where $Q_{2}=R_{0}\left( ad^{\ast }-bc^{\ast }\right) R_{1}^{\ast }$, with $%
\det R_{0}\neq 0$, $\det R_{1}\neq 0$.

Then, formula \eqref{t} is equivalent to $\det Q_{2}=0$.

Notice that, if $\Phi $ is of the form
\begin{equation}
\Phi =L_{0}+L_{1}\Phi _{0}L_{2}+L_{3}\Phi _{0}^{\ast }L_{4}\text{,}
\label{u}
\end{equation}
then 
\begin{align*}
\left( L_0\right) _{11} &= R_0\left( \beta _0\alpha _1^{*}-\alpha _0\beta
_1^{*}+\delta _0\gamma _1^{*}-\gamma _0\delta _1^{*}\right) R_1^{*}\text{,}
\\
\left( L_1\right) _{11} &= R_0\delta _0\text{, \ }\left( L_1\right)
_{12}=R_0\gamma _0\text{,} \\
\left( L_2\right) _{11} &= -\alpha _1^{*}R_1^{*}\text{, \ }\left( L_2\right)
_{21}=\beta _1^{*}R_1^{*}\text{,} \\
\left( L_3\right) _{11} &= R_0\alpha _0\text{, \ }\left( L_3\right)
_{12}=-R_0\beta _0\text{,} \\
\left( L_4\right) _{11} &= \delta _1^{*}R_1^{*}\text{, \ }\left( L_4\right)
_{21}=\gamma _1^{*}R_1^{*}\text{.}
\end{align*}
As $\alpha _{0}\beta _{1}^{\ast }+\delta _{0}\gamma _{1}^{\ast }=\beta
_{0}\alpha _{1}^{\ast }+\gamma _{0}\delta _{1}^{\ast }$, one obtains 
\begin{equation*}
\left( L_{0}\right) _{11}=2R_{0}\left( \beta _{0}\alpha _{1}^{\ast }-\alpha
_{0}\beta _{1}^{\ast }\right) R_{1}^{\ast }=2R_{0}\left( \delta _{0}\gamma
_{1}^{\ast }-\gamma _{0}\delta _{1}^{\ast }\right) R_{1}^{\ast }\text{.}
\end{equation*}

The main problem here involved is to discover conditions over the matrices $%
L_{0}$, $L_{1}$, $L_{2}$, $L_{3}$ and $L_{4}$, so that $\Phi $ is symplectic
for every symplectic matrix $\Phi _{0}$. More generally, the problem is to
discover conditions over $\Phi $, with $Q_{2}=R_{0}\left( ad^{\ast
}-bc^{\ast }\right) R_{1}^{\ast }$, such that $\Phi $ is symplectic for
every symplectic matrix $\Phi _{0}$. These questions can be completely
solved in dimension one as it is done in the Appendix.

Let us take a look to simple cases in dimension greater than one.

Assume that $L_{0}=L_{3}=L_{4}=0$ and that $L_{1}$ and $L_{2}$ are both
symplectic or antisymplectic. Then $\Phi $ is symplectic for every
symplectic matrix $\Phi _{0}$. The same happens, \textit{mutatis mutandis},
when $L_{0}=L_{1}=L_{2}=0$.

The purpose of this article is to use the polar representation theorem in
order to obtain results on the Sturm-Liouville problem.

\section{A theorem on two parameters dependent\\
symplectic matrices}

\noindent 
In this section we prove a theorem that we shall need later and is a good
introduction to the method we use in this article.

As before, let $\tau \in \left[ 0,t\right] \subset \left[ 0,T\right[ $ and $%
\lambda \in \left] l_{-1},l_{1}\right[ \subset \mathbb{R}$. Consider the $C^{1}$
function $\left( \tau ,\lambda \right) \mapsto \Phi \left( \tau ,\lambda
\right) $, where $\Phi \left( \tau ,\lambda \right) $ is symplectic.

In the following we shall denote $\frac \partial {\partial \lambda }(
\cdot) \equiv ( \cdot )^{\prime }$ the eigenvalue
derivative, the derivative with respect to $\lambda $.

We define 
\begin{equation*}
M_1=\dot{\Phi}\Phi ^{-1}\text{, \ }S_1=-JM_1\text{.}
\end{equation*}
and 
\begin{equation*}
M_2=\Phi ^{\prime }\Phi ^{-1}\text{, \ }S_2=-JM_2\text{.}
\end{equation*}

Notice that, as $\Phi $, $M_j$ and $S_j$ are both time and eigenvalue
dependent, we shall use, as we did already before, the notations $\Phi
\equiv \Phi \left( \tau \right) \equiv \Phi \left( \tau ,\lambda \right) $, $%
M_j\equiv M_j\left( \tau \right) \equiv M_j\left( \tau ,\lambda \right) $, $%
S_j\equiv S_j\left( \tau \right) \equiv S_j\left( \tau ,\lambda \right) $,
and so on ($j=1,2$). We also naturally denote
\begin{equation*}
\Phi =
\begin{bmatrix}
Q_2 & Q_1 \\ 
P_2 & P_1
\end{bmatrix}
\text{,\quad}
S_j=
\begin{bmatrix}
A_j & B_j^{*} \\ 
B_j & C_j
\end{bmatrix}
\text{,}
\end{equation*}
and assume that $C_1$ and $C_2$ are $C^1$ functions.

Let $\epsilon _{1}=\pm 1$, $\epsilon _{2}=\pm 1$, $\epsilon =\epsilon
_{1}\epsilon _{2}$.

Let $\tau _{0}\geq 0$ and $\chi :\left] \tau _{0},T\right[ \rightarrow
\left] l_{-1},l_{1}\right[ $ a continuous function, such that $\epsilon \chi 
$ is strictly decreasing and $\lim_{\tau \rightarrow T}\epsilon \chi \left(
\tau \right) =\epsilon l_{0}\geq \epsilon l_{-\epsilon }$ and $\lim_{\tau
\rightarrow \tau _{0}}\chi \left( \tau \right) =l_{\epsilon _{{}}}$.

Assume that 
\begin{equation}
\det Q_{2}\left( \tau ,\lambda \right) =0\Rightarrow \epsilon \left( \lambda
-\chi \left( \tau \right) \right) >0\text{,}  \label{e}
\end{equation}
and that
\begin{equation}
\epsilon \left( \lambda -\chi \left( \tau \right) \right) >0\Rightarrow
\left\{ \epsilon _{1}C_{1}\left( \tau ,\lambda \right) >0\wedge \epsilon
_{2}C_{2}\left( \tau ,\lambda \right) >0\right\} .  \label{f}
\end{equation}

\begin{theorem}
\label{T2}Under Conditions \eqref{e} and \eqref{f}, equation 
\begin{equation*}
\det Q_2\left( \tau ,\lambda \right) =0\text{,}
\end{equation*}
defines implicitly $n$ sets of continuous functions $\tau \mapsto \lambda
_{jk}\left( \tau \right) $, ($j=1,2,\ldots ,n$), with the index $k\in \mathbb{Z}
$ and bounded below. Some of these sets, or all, may be empty. In each
nonempty set these functions have a natural order: $\epsilon \lambda
_{jk}\left( \tau \right) <\epsilon \lambda _{j,k+1}\left( \tau \right)
<\epsilon \lambda _{j,k+2}\left( \tau \right) <\cdots $.

Let $l\in \left] l_{-1},l_1\right[ $ and $t\in \left[ 0,T\right[ $, and
assume that $\det Q_2\left( t,l\right) \neq 0$. Denote by $\mu _j$ the
cardinal of the set $\left\{ k\in \mathbb{N}:\epsilon \left( \lambda
_{jk}\left( t\right) -l\right) <0\right\} $ and let $\mu =\sum_{j=1}^n\mu _j$%
. Then, $\mu $ is the number of times, counting the multiplicities, that $%
Q_2\left( \tau ,l\right) $ is singular, for $\tau <t$.
\end{theorem}

\begin{proof}
As the proof for $\epsilon =-1$ is similar, suppose that $\epsilon =1$.
Define 
\begin{equation*}
\mathcal{D}\mathbb{=}\left\{ \left( \tau ,\lambda \right) :\tau \in \left] \tau
_0,T\right[ ,\lambda \in \left] l_{-1},l_1\right[ ,\lambda >\chi \left( \tau
\right) \right\}
\end{equation*}

From Theorem \ref{T1}, one has that 
\begin{equation*}
Q_1( \tau ,\lambda ) =r( \tau ,\lambda ) \sin \varphi
( \tau ,\lambda ) \text{,}
\end{equation*}
\begin{equation*}
Q_2( \tau ,\lambda ) =r( \tau ,\lambda ) \cos \varphi
( \tau ,\lambda ) \text{,}
\end{equation*}
where $r( \tau ,\lambda ) $, $\varphi ( \tau ,\lambda
) $, for $( \tau ,\lambda ) \in \mathcal{D}$, are $C^1$
matrix-valued functions such that $\det r( \tau ,\lambda ) \neq 0$
and $\varphi ( \tau ,\lambda ) $ is symmetric for every $(
\tau ,\lambda ) $ and the eigenvalues of $\varphi $ are $C^1$
functions of $\tau $ and $\lambda $. Denote $\varphi _1( \tau ,\lambda
) ,\linebreak[0]\ldots ,\varphi _n( \tau ,\lambda ) $ such eigenvalues.
Then $\epsilon _1\dot{\varphi}_1( \tau ,\lambda ) ,\linebreak[0]\ldots
,\linebreak[0]\epsilon _1\dot{\varphi}_n( \tau ,\lambda ) $ and
$\epsilon _2\varphi _1^{^{\prime }}( \tau
,\lambda ) ,\linebreak[0]\ldots ,\linebreak[0]
\epsilon _2\varphi _n^{^{\prime }}( \tau,\lambda ) $
are positive continuous functions, for $( \tau
,\lambda ) \in \mathcal{D}$. The matrix $Q_2( \tau ,l) $,
with $\tau <t$, is singular if, with $\lambda =l$, 
\begin{equation}
\varphi _j( \tau ,\lambda ) =\frac \pi 2+k\pi \text{,}
\label{J13}
\end{equation}
for some $j=1,\ldots ,n$ and $k\in \mathbb{Z}$.

Notice that $\varphi _j( \tau ,\lambda ) >\varphi _j(
0,\lambda ) $, so that the set of possible $k$ either is empty or is
bounded below.

Consider the sets $\Lambda _{jk}$ defined by equation \eqref{J13}: 
\begin{equation*}
\Lambda _{jk}=\left\{ \left( \tau ,\lambda \right) \in \mathcal{D}:\varphi
_j\left( \tau ,\lambda \right) =\frac \pi 2+k\pi \right\} \text{,}
\end{equation*}
If one of the sets $\Lambda _{jk}$ is not empty, then, locally, it defines a
function $\lambda _{jk}\left( \tau \right) $, and 
\begin{equation*}
\frac{d\lambda _{jk}}{d\tau }\left( \tau \right) =-\frac{\partial \varphi _j%
}{\partial \tau }\left( \tau ,\lambda _{jk}\left( \tau \right) \right)
\left( \frac{\partial \varphi _j}{\partial \lambda }\left( \tau ,\lambda
_{jk}\left( \tau \right) \right) \right) ^{-1}\text{,}
\end{equation*}
because $\epsilon _1/\epsilon _2=1$.

Therefore, $\dot{\lambda}_{jk}\left( \tau \right) <0$. Hence, the sets $%
\Lambda _{jk}$ defined by \eqref{J13} are totally ordered: $\left( \tau
_1,\lambda _1\right) \succ \left( \tau _2,\lambda _2\right) $ if $\tau
_1>\tau _2$ and $\lambda _1<\lambda _2$. $\Lambda _{jk}$ has an infimum $%
\left( t_{jk},l_{jk}\right) $. The case $t_{jk}>0$ and $l_{jk}<l_1$ can not
happen from the implicit function theorem. The case $t_{jk}=0$ and $%
l_{jk}<l_1$ is impossible as formula \eqref{e} makes clear. Hence, $%
t_{jk}\geq 0$ and $l_{jk}=l_1$.

Hence, $\lambda _{jk}$ are $C^1$ functions $\lambda _{jk}\left( \tau \right)
:\left] t_{jk},T\right[ \rightarrow \mathbb{R}$, such that 
\begin{equation*}
\lim_{\tau \rightarrow t_{jk}}\lambda _{jk}\left( \tau \right) =l_1\text{ , }%
\frac d{d\tau }\lambda _{jk}\left( \tau \right) <0\text{ , }\varphi _j\left(
\tau ,\lambda _{jk}\left( \tau \right) \right) =\frac \pi 2+k\pi \text{.}
\end{equation*}

We remark that, namely from \eqref{e}, we have 
\begin{equation*}
\lambda _{j,k+1}\left( \tau \right) >\lambda _{jk}\left( \tau \right) >\chi
\left( \tau \right) \text{.}
\end{equation*}
Hence, one has that the following three assertions are equivalent:

a) There is a $\tau <t$, such that $\lambda _{jk}\left( \tau \right) =l$.

b) There is a $\tau <t$, such that $\varphi _j\left( \tau ,l\right) =\frac
\pi 2+k\pi $.

c) $\lambda _{jk}\left( t\right) <l$.

From this, the theorem follows.
\end{proof}

\section{Some formulas}

\noindent As before, let $\tau \in \left[ 0,t\right] \subset \left[ 0,T\right[ $ and $%
\lambda \in \left] l_{-1},l_{1}\right[ \subset \mathbb{R}$. Consider the $C^{1}$
function $\left( \tau ,\lambda \right) \mapsto \Phi \left( \tau ,\lambda
\right) $, where $\Phi \left( \tau ,\lambda \right) $ is symplectic. We
define 
\begin{equation*}
M_{1}=\dot{\Phi}\Phi ^{-1}\text{,\quad}S_{1}=-JM_{1}\text{.}
\end{equation*}

Notice that, as $\Phi $, $M_1$ and $S_1$ are both time and eigenvalue
dependent, we shall use, as we did already before, the notations $\Phi
\equiv \Phi \left( \tau \right) \equiv \Phi \left( \tau ,\lambda \right) $, $%
M_1\equiv M_1\left( \tau \right) \equiv M_1\left( \tau ,\lambda \right) $, $%
S_1\equiv S_1\left( \tau \right) \equiv S_1\left( \tau ,\lambda \right) $,
and so on.

In the following we shall denote $\frac \partial {\partial \lambda }
(\cdot) \equiv (\cdot)^{\prime }$ the eigenvalue
derivative, the derivative with respect to $\lambda $.

It is now natural to compute $\Phi ^{\prime }$ and $\Phi ^{\prime }\Phi
^{-1}\equiv M_2$.

Deriving both members of $\dot{\Phi}=M_{1}\Phi $ in order to $\lambda $, one
obtains 
\begin{equation}
\dot{\Phi}^{\prime }=M_{1}^{\prime }\Phi +M_{1}\Phi ^{\prime }\text{.}
\label{g}
\end{equation}

We shall use now the variations of parameters method. Write $\Phi ^{\prime
}=\Phi K$, where $K$ is both time and eigenvalue dependent: $K\equiv K\left(
\tau ,\lambda \right) $.

Let $K_{0}=K\left( 0,\lambda \right) \equiv K\left( 0\right) $. As $K\left(
0,\lambda \right) =\Phi ^{-1}\left( 0\right) \Phi ^{\prime }\left( 0\right) $,
and 
\begin{equation*}
\Phi \left( \tau \right) =\Phi \left( 0\right) +\int_{0}^{\tau
_{{}}}M_{1}\left( \sigma \right) \Phi \left( \sigma \right) d\sigma \text{,}
\end{equation*}
one has 
\begin{equation*}
\Phi ^{\prime }\left( \tau \right) =\left( \Phi \left( 0\right) \right)
^{\prime }+\int_{0}^{\tau _{{}}}\left( M_{1}\left( \sigma \right) \Phi
\left( \sigma \right) \right) ^{\prime }d\sigma \text{.}
\end{equation*}
Hence, $\Phi ^{\prime }\left( 0\right) =\left( \Phi \left( 0\right) \right)
^{\prime }$ and $K_0=K\left( 0,\lambda \right) =\Phi ^{-1}\left( 0\right)
\left( \Phi \left( 0\right) \right) ^{\prime }$.

On the other hand, one obtains 
\begin{equation}
\dot{\Phi}^{\prime }=\dot{\Phi}K+\Phi \dot{K}=M_{1}\Phi K+\Phi \dot{K}%
=M_{1}\Phi ^{\prime }+\Phi \dot{K}\text{.}  \label{h}
\end{equation}

Comparing \eqref{g} with \eqref{h}, one has 
\begin{equation*}
M_{1}^{\prime }\Phi =\Phi \dot{K}\text{.}
\end{equation*}
From this one concludes that $\dot{K}=\Phi ^{-1}M_{1}^{\prime }\Phi $.
Therefore 
\begin{equation*}
K\left( \tau \right) =K_{0}+\int_{0}^{\tau _{{}}}\Phi ^{-1}\left( \sigma
\right) M_{1}^{\prime }\left( \sigma \right) \Phi \left( \sigma \right)
d\sigma \text{.}
\end{equation*}

From now on we shall use the notations: 
\begin{equation*}
F\left( \tau ,\sigma \right) =\Phi \left( \tau \right) \Phi ^{-1}\left(
\sigma \right) \text{,\quad}F_{0}\left( \tau ,\sigma \right) =\Phi
_{0}\left( \tau \right) \Phi _{0}^{-1}\left( \sigma \right) \text{.}
\end{equation*}

Then
\begin{equation*}
\begin{split}
M_{2}\left( \tau \right) &\equiv
\Phi ^{\prime }\Phi ^{-1}=\Phi K\Phi ^{-1}\\
&=\Phi \left( \tau \right) \Phi ^{-1}\left( 0\right) \left( \Phi \left(
0\right) \right) ^{\prime }\Phi ^{-1}\left( \tau \right) \\
&\quad +\int_{0}^{\tau _{{}}}F\left( \tau ,\sigma \right) M_{1}^{\prime }\left(
\sigma \right) \Phi \left( \sigma \right) F^{-1}\left( \tau ,\sigma \right)
d\sigma \text{.}
\end{split}
\end{equation*}

Notice that, if $V$ is any $2n\times 2n$ eigenvalue dependent matrix, 
\begin{multline*}
\int_{0}^{\tau _{{}}}\Phi ^{-1}\left( \sigma \right) VM_{1}\left( \sigma
\right) \Phi \left( \sigma \right) d\sigma =\int_{0}^{\tau _{{}}}\Phi
^{-1}\left( \sigma \right) V\dot{\Phi}\left( \sigma \right) d\sigma \\
=\left[ \Phi ^{-1}\left( \sigma \right) V\Phi \left( \sigma \right) \right]
_{0}^{\tau _{{}}}+\int_{0}^{\tau _{{}}}\Phi ^{-1}\left( \sigma \right)
M_{1}\left( \sigma \right) V\Phi \left( \sigma \right) d\sigma \text{.}
\end{multline*}

Hence, 
\begin{equation*}
\begin{split}
M_{2}\left( \tau \right) &= \Phi \left( \tau \right) \left( \Phi ^{-1}\left(
0\right) \left( \Phi \left( 0\right) \right) ^{\prime }+\left[ \Phi
^{-1}\left( \sigma \right) V\Phi \left( \sigma \right) \right] _{0}^{\tau
_{{}}}\right) \Phi ^{-1}\left( \tau \right) \\
&\quad+\int_{0}^{\tau _{{}}}F\left( \tau ,\sigma \right) G_{1}F^{-1}\left( \tau
,\sigma \right) d\sigma \text{,}
\end{split}
\end{equation*}
with 
\begin{equation}
G_{1}\equiv M_{1}^{\prime }\left( \sigma \right) -VM_{1}\left( \sigma
\right) +M_{1}\left( \sigma \right) V  \label{v}
\end{equation}
or, equivalently, 
\begin{equation*}
\begin{split}
M_{2}\left( \tau \right) &=V+\Phi \left( \tau \right) \Phi ^{-1}\left(
0\right) \left( \left( \Phi \left( 0\right) \right) ^{\prime }-V\Phi \left(
0\right) \right) \Phi ^{-1}\left( \tau \right) + \\
&\quad+\int_{0}^{\tau _{{}}}F\left( \tau ,\sigma \right) G_{1}F^{-1}\left( \tau
,\sigma \right) d\sigma \text{.}
\end{split}
\end{equation*}

Choosing 
\begin{equation}
V=\left( \Phi \left( 0\right) \right) ^{\prime }\Phi ^{-1}\left( 0\right) 
\text{,}  \label{c}
\end{equation}
one has 
\begin{equation}
M_{2}\left( \tau \right) =V+\int_{0}^{\tau _{{}}}F\left( \tau ,\sigma
\right) G_{1}F^{-1}\left( \tau ,\sigma \right) d\sigma \text{,}  \label{d}
\end{equation}
with $V$ defined by \eqref{c} and $G_{1}$\ defined by \eqref{v}.

Equation \eqref{d} can be written 
\begin{equation*}
\begin{split}
M_2\left( \tau \right) &=\left( \Phi \left( 0\right) \right) ^{\prime }\Phi
^{-1}\left( 0\right) \\
&\quad +\int_0^{\tau _{}}F\left( \tau ,\sigma \right) G_2F^{-1}\left( \tau
,\sigma \right) d\sigma \text{,}
\end{split}
\end{equation*}
with 
\begin{equation*}
G_2\equiv \Phi \left( 0\right) \left( \Phi ^{-1}\left( 0\right) M_1\left(
\sigma \right) \Phi \left( 0\right) \right) ^{\prime }\Phi ^{-1}\left(
0\right) \text{.}
\end{equation*}

\section{First remarkable case}

\noindent Let us take 
\begin{equation*}
\Phi =L_1\Phi _0L_2\text{,}
\end{equation*}
where 
\begin{equation*}
\dot{\Phi}_0=M_0\Phi _0\text{, \ }M_0=-JS_0\text{.}
\end{equation*}

$L_{1}$ and $L_{2}$ are both symplectic or both antisymplectic and
eigenvalue dependent: $L_{1}\equiv L_{1}\left( \lambda \right) $, $%
L_{2}\equiv L_{2}\left( \lambda \right) $. As before, $\Phi $, $\Phi _{0}$, $%
M_{0}$ and $S_{0}$ are both time and eigenvalue dependent: $\Phi \equiv \Phi
\left( \tau \right) \equiv \Phi \left( \tau ,\lambda \right) $, $\Phi
_{0}\equiv \Phi _{0}\left( \tau \right) \equiv \Phi _{0}\left( \tau ,\lambda
\right) $, $M_{0}\equiv M_{0}\left( \tau \right) \equiv M_{0}\left( \tau
,\lambda \right) $, $S_{0}\equiv S_{0}\left( \tau \right) \equiv S_{0}\left(
\tau ,\lambda \right) $ and so on.

As $\dot{\Phi}=L_1\dot{\Phi}_0L_2=L_1M_0\Phi _0L_2=L_1M_0L_1^{-1}\Phi $, one
has 
\begin{align*}
M_1&=L_1M_0L_1^{-1}\text{,}\\
K_0&=L_2^{-1}L_1^{-1}\left( L_1L_2\right) ^{\prime }\text{.}
\end{align*}

Then 
\begin{equation*}
\begin{split}
M_{2}\left( \tau \right) &=L_{1}\Phi _{0}\left( \tau \right)
L_{1}^{-1}\left( L_{1}L_{2}\right) ^{\prime }L_{2}^{-1}\Phi _{0}^{-1}\left(
\tau \right) L_{1}^{-1} \\
&\quad +\int_{0}^{\tau _{{}}}F\left( \tau ,\sigma \right) M_{1}^{\prime }\left(
\sigma \right) F^{-1}\left( \tau ,\sigma \right) d\sigma \text{,}
\end{split}
\end{equation*}
and 
\begin{equation*}
M_{2}\left( \tau \right) =V+\int_{0}^{\tau _{{}}}F\left( \tau ,\sigma
\right) G_{3}F^{-1}\left( \tau ,\sigma \right) d\sigma \text{,}
\end{equation*}
where 
\begin{equation*}
V=\left( L_{1}L_{2}\right) ^{\prime }\left( L_{1}L_{2}\right) ^{-1}\text{,}
\end{equation*}
and 
\begin{equation*}
G_{3}\equiv M_{1}^{\prime }\left( \sigma \right) -VM_{1}\left( \sigma
\right) +M_{1}\left( \sigma \right) V\text{.}
\end{equation*}

One also has the formula 
\begin{equation}
M_{2}\left( \tau \right) =V+\int_{0}^{\tau _{{}}}L_{1}F_{0}\left( \tau
,\sigma \right) G_{4}F_{0}^{-1}\left( \tau ,\sigma \right) L_{1}^{-1}d\sigma
,  \label{i}
\end{equation}
where 
\begin{equation*}
G_{4}\equiv M_{0}^{\prime }+M_{0}L_{2}^{\prime }L_{2}^{-1}-L_{2}^{\prime
}L_{2}^{-1}M_{0}\text{.}
\end{equation*}

\begin{remark}
If $\left( L_1\right) _{12}=0$, $\det \left( \left( L_1\right) _{11}\right)
\neq 0$ and $C_0>0$ ($C_0<0$), then $C_1=\left( L_1\right) _{11}C_0\left(
L_1\right) _{11}^{*}>0$ ($<0$).
\end{remark}

\subsection{Example: the Morse index theorem}
\ \\
\noindent
Let $N$ a symmetric $n\times n$ matrix. Define $Q_{1}=Q_{s}$ and 
$Q_{2}=Q_{c}+Q_{s}N$. Then $Q_{1}$ and $Q_{2}$ are isotropic, $W=I$. Hence,
from Theorem \ref{T1}, one has that 
\begin{equation*}
Q_{1}\left( \tau \right) =Q_{s}\left( \tau \right) =r\left( \tau \right)
\sin \varphi \left( \tau \right) \text{,}
\end{equation*}
\begin{equation}
Q_{2}\left( \tau \right) =Q_{c}\left( \tau \right) +Q_{s}\left( \tau \right)
N=r\left( \tau \right) \cos \varphi \left( \tau \right) \text{,}  \label{J12}
\end{equation}
where $r\left( \tau \right) $, $\varphi \left( \tau \right) $, for $\tau \in
\left[ 0,T\right[ $, are $C^{1}$ matrix-valued functions such that $\det
r\left( \tau \right) \neq 0$ and $\varphi \left( \tau \right) $ is symmetric
for every $\tau $ and the eigenvalues of $\varphi $ are $C^{1}$ functions of 
$\tau $. Denote $\varphi _{1}\left( \tau \right) ,\ldots ,\varphi _{n}\left(
\tau \right) $ such eigenvalues, with $\varphi _{j}\left( 0\right) =0$. Then 
$\dot{\varphi}_{1}\left( \tau \right) ,\ldots ,\dot{\varphi}_{n}\left( \tau
\right) $ are positive continuous functions.

Let $t\in \left[ 0,T\right[ $. Assume that $Q_{2}\left( t\right) $ is
invertible and that $\varphi _{j}\left( 0\right) =0$, $j=1,\ldots ,n$, and
define $\mu _{j}\in \mathbb{Z}$, such that 
\begin{equation*}
-\frac{\pi }{2}+\mu _{j}\pi <\varphi _{j}\left( t\right) <\frac{\pi }{2}+\mu
_{j}\pi \text{.}
\end{equation*}

Define the index $\mu $: 
\begin{equation}
\mu =\sum_{j=1}^n\mu _j\text{.}  \label{J9}
\end{equation}
Then, $\mu $ is the number of times that $Q_{2}\left( \tau \right) $ is
singular, for $\tau \in \left[ 0,t\right] $, taking into account the
multiplicity of the singularity, i.e. the dimension of $\ker Q_{2}$.

Consider now the Lagrangian

\begin{equation*}
L\left( q,\dot{q},\tau \right) =\frac{1}{2}\left( \dot{q},C\left( \tau
\right) ^{-1}\dot{q}\right) -\left( \dot{q},C\left( \tau \right)
^{-1}B\left( \tau \right) q\right) -\frac{1}{2}\left( q,\mathcal{A}\left(
\tau \right) q\right) \text{,}
\end{equation*}
where $\mathcal{A}=A-B^{\ast }C^{-1}B$.

Consider now the real separable Hilbert space $\mathcal{H}$, whose elements
are the continuous functions $\gamma :\left[ 0,t\right] \rightarrow \mathbb{R}%
^{n}$, 
\begin{equation*}
\gamma \left( \tau \right) =-\int_{\tau _{{}}}^{t_{{}}}\dot{\gamma}\left(
\sigma \right) d\sigma \text{,}
\end{equation*}
for $\dot{\gamma}\in L^{2}\left( \left[ 0,t\right] ;\mathbb{R}^{n}\right) $.
The inner product $\left\langle .,.\right\rangle $ in $\mathcal{H}$ is
defined by 
\begin{equation*}
\left\langle \gamma _{1},\gamma _{2}\right\rangle =\int_{0}^{t}\left( \dot{%
\gamma}_{1}\left( \tau \right) ,C\left( \tau \right) ^{-1}\dot{\gamma}%
_{2}\left( \tau \right) \right) d\tau \text{.}
\end{equation*}
One denotes $\left\langle \gamma ,\gamma \right\rangle =\left\| \gamma
\right\| ^2$.

To the Lagrangian $L$ corresponds the action 
\begin{equation*}
\mathcal{S}\left( \gamma \right) =\int_{0}^{t}L\left( \gamma \left( \tau
\right) ,\dot{\gamma}\left( \tau \right) ,\tau \right) d\tau +\frac{1}{2}%
\left( \gamma \left( 0\right) ,N\gamma \left( 0\right) \right) \text{,}
\end{equation*}
where $N$, as before, is a symmetric $n\times n$ matrix.

The quadratic form $\mathcal{S}:\mathcal{H}\rightarrow \mathbb{R}$, defines a
symmetric operator $\mathcal{L}\left( t\right) \equiv \mathcal{L}:\mathcal{%
H\rightarrow H}$, $\mathcal{S}\left( \gamma \right) =\frac{1}{2}\left\langle
\gamma ,\mathcal{L}\gamma \right\rangle $, 
\begin{equation*}
\begin{split}
\left\langle \gamma _{1},\mathcal{L}\gamma _{2}\right\rangle
&=\int_{0}^{t}\left( \dot{\gamma}_{1}\left( \tau \right) ,C\left( \tau
\right) ^{-1}\dot{\gamma}_{2}\left( \tau \right) \right) d\tau \\
&\quad -\int_{0}^{t}\left( \dot{\gamma}_{1}\left( \tau \right) ,C\left( \tau
\right) ^{-1}B\left( \tau \right) \gamma _{2}\left( \tau \right) \right)
d\tau \\
&\quad -\int_{0}^{t}\left( \dot{\gamma}_{2}\left( \tau \right) ,C\left( \tau
\right) ^{-1}B\left( \tau \right) \gamma _{1}\left( \tau \right) \right)
d\tau \\
&\quad 
-\int_{0}^{t}\left( \gamma _{1}\left( \tau \right) ,\mathcal{A}\left( \tau
\right) \gamma _{2}\left( \tau \right) \right) d\tau 
+\left( \gamma_{1}\left( 0\right) ,N\gamma _{2}\left( 0\right) \right) \text{,}
\end{split}
\end{equation*}
which has the following expression
\begin{equation*}
\begin{split}
\left( \mathcal{L}\gamma \right) \left( \tau \right) &=\gamma \left( \tau
\right) +\int_{\tau _{{}}}^{t_{{}}}B\left( \sigma \right) \gamma \left(
\sigma \right) d\sigma \\
&\quad -\int_{\tau _{{}}}^{t_{{}}}C\left( \sigma \right) d\sigma
\int_{0_{{}}}^{\sigma _{{}}}B^{\ast }\left( \theta \right) C\left( \theta
\right) ^{-1}\dot{\gamma}\left( \theta \right) d\theta \\
&\quad -\int_{\tau _{{}}}^{t_{{}}}C\left( \sigma \right) d\sigma
\int_{0_{{}}}^{\sigma _{{}}}\mathcal{A}\left( \theta \right) \gamma \left(
\theta \right) d\theta +\int_{\tau _{{}}}^{t_{{}}}C\left( \sigma \right)
d\sigma N\gamma \left( 0\right) \text{.}
\end{split}
\end{equation*}

$\mathcal{L}$ is the sum of four symmetric operators. The first one is the
identity. The second one, which involves $B$, is a Hilbert-Schmidt operator.
The third one, which involves $\mathcal{A}$, is a trace class operator. The
forth one, which involves $N$, is a finite rank operator.

The eigenvalues $\lambda $ of $\mathcal{L}$ are given by the equation 
\begin{equation}
\mathcal{L}\gamma =\lambda \gamma \text{,\quad}\gamma \in \mathcal{H}
\text{,\quad}\gamma \neq 0\text{.}  \label{J3}
\end{equation}

Assume that $\lambda \neq 1$ and put $\varepsilon =\left( 1-\lambda \right)
^{-1}$. As $\frac{d\varepsilon }{d\lambda }=\left( 1-\lambda \right) ^{-2}>0$%
, we shall use $\varepsilon $ instead of $\lambda $ as a parameter, and $%
\left( \cdot \right) ^{\prime }\equiv \frac{\partial }{\partial \varepsilon }%
\left( \cdot \right) $.

Then, one has 
\begin{equation}
\left| \varepsilon \right| >\left( at+bt^{2}\right) ^{-1}\text{,}  \label{x}
\end{equation}
where $a,b>0$ (see \cite{rezende4}).

Define 
\begin{align*}
A_1&=\varepsilon A+\left( \varepsilon ^2-\varepsilon \right)
B^{*}C^{-1}B=\varepsilon \mathcal{A}+\varepsilon ^2B^{*}C^{-1}B\\
B_1&=\varepsilon B,\quad
C_1=C\text{.}
\end{align*}

Call $\mathcal{L}_{\varepsilon _{{}}}$ the operator $\mathcal{L}$ where one
puts $A_{1}$, $B_{1}$, $C_{1}$ and $\varepsilon N$ instead of $A$, $B$, $C$
and $N$. Notice that $\mathcal{L}=\mathcal{L}_{1}$. Then equation \eqref{J3}
becomes 
\begin{equation*}
\mathcal{L}_{\varepsilon _{{}}}\gamma =0\text{,\quad}\gamma \in \mathcal{H}%
\text{ , }\gamma \neq 0\text{.}
\end{equation*}

This equation can be rewritten 
\begin{gather*}
\dot{\gamma}=B_{1}\gamma +C_{1}\beta \text{,\quad}\dot{\beta}=-A_{1}\gamma
-B_{1}^{\ast }\beta ,\\
\gamma \left( t\right) =0\text{,\quad}\beta \left( 0\right) -\varepsilon
N\gamma \left( 0\right) =0\text{.}
\end{gather*}

Put $L_{1}=I_{2n}$ and 
\begin{equation*}
L_{2}=
\begin{bmatrix}
fI_{n} & kfI_{n} \\ 
\varepsilon fN & f^{-1}I_{n}+k\varepsilon fN
\end{bmatrix}
\text{,}
\end{equation*}
where $k$ is constant and $f\equiv f\left( \varepsilon \right) \neq 0$.

Then $\Phi =L_{1}\Phi _{0}L_{2}=\Phi _{0}L_{2}$. Put $\Phi
_{11}=Q_{\varepsilon ,2}$, $\Phi _{12}=Q_{\varepsilon ,1}$ and so on. Hence, 
$Q_{\varepsilon ,2}=f\left( Q_{c}+\varepsilon Q_{s}N\right) $ and $%
Q_{2}=f^{-1}Q_{1,2}$.

Then $\left( L_{2}^{\prime }L_{2}^{-1}\right) _{12}=0$, and if $f+2f^{\prime
}\varepsilon =0$, 
\begin{equation*}
\left( L_{2}^{\prime }L_{2}^{-1}\right) _{22}=-\left(
L_{2}^{\prime }L_{2}^{-1}\right) _{11}=\left( 2\varepsilon \right) ^{-1},\quad
\left( L_{2}^{\prime }L_{2}^{-1}\right) _{21}=0.
\end{equation*}

Now, one computes $G_{4}$: 
\begin{equation*}
M_{0}^{\prime }+M_{0}L_{2}^{\prime }L_{2}^{-1}-L_{2}^{\prime
}L_{2}^{-1}M_{0}=
\begin{bmatrix}
B & \varepsilon ^{-1}C \\ 
-\varepsilon B^{\ast }C^{-1}B & -B^{\ast }
\end{bmatrix}
\text{.}
\end{equation*}

Denoting 
\begin{equation*}
\begin{bmatrix}
X & Z \\ 
W & Y
\end{bmatrix}
=\Phi _{0}\left( \tau \right) \Phi _{0}^{-1}\left( \sigma \right) 
\text{,}
\end{equation*}
one has 
\begin{equation*}
C_{2}=\varepsilon ^{-1}\int_{0}^{\tau _{{}}}\left( XC-\varepsilon ZB^{\ast
}\right) C^{-1}\left( CX^{\ast }-\varepsilon BZ^{\ast }\right) d\sigma \text{%
.}
\end{equation*}
Then $\varepsilon C_{2}>0$ for $\tau >0$.

From this, from \eqref{x} and from Theorem \ref{T2}\ one can easily state
the following theorem, whose complete proof can be seen in detail in \cite
{rezende4}.

\begin{theorem}
Let $\lambda \left( t\right) $ be an eigenvalue of the operator $\mathcal{L}%
\left( t\right) $. Then, there are three possibilities: 1) $\lambda \left(
t\right) =1$; 2) (and 3)) $\lambda \left( t\right) >1$ ($\lambda \left(
t\right) <1$); in this case there exists a $t_0\geq 0$ and a continuous
function $\lambda \left( \tau \right) $, for $\tau \in \left[ t_0,t\right] $%
, such that $\lambda \left( \tau \right) $ is an eigenvalue of the operator $%
\mathcal{L}\left( \tau \right) $ and $\lambda \left( t_0\right) =1$;
moreover, $\lambda \left( \tau \right) $ is $C^1$ in $\left] t_0,t\right] $
with $\dot{\lambda}\left( \tau \right) >0$ ($\dot{\lambda}\left( \tau
\right) <0$).

The eigenvalues of $\mathcal{L}\left( t\right) $ which are different from $1$
can be organized in $2n$ sets; $n$ for those $>1$, $n$ for those $<1$. Some
of these sets may be empty. In each set, the eigenvalues have a natural
order: $\lambda _0\left( \tau \right) >\lambda _1\left( \tau \right) >\cdots
>1$, or $\lambda _0\left( \tau \right) <\lambda _1\left( \tau \right)
<\cdots <1$, for every $\tau $. In particular, the eigenspace of $\lambda
\neq 1$ has at most dimension $n$.

Let $Q_2\equiv Q_c+Q_sN$, be a solution of the system \eqref{J1}. Then, $%
Q_2\left( t\right) $ is invertible if and only if $\mathcal{L}\left(
t\right) $ is invertible and the number of the negative eigenvalues of $%
\mathcal{L}$ (its Morse index) is $\mu $, as defined by \eqref{J9}.
\end{theorem}

\subsection{Example}
\ \\
\noindent Let $A_{0}=( 1-\mu ) A_{3}+\mu A_{4}$, $B_{0}=( 1-\mu
) B_{3}+\mu B_{4}$, $C_{0}=( 1-\mu ) C_{3}+\mu C_{4}$.
Assume that $A_{3}$, $A_{4}$, $B_{3}$, $B_{4}$, $C_{3}$, $C_{4}$, $L_{1}$
and $L_{2}$ are $\mu $-independent and that $L_{1}$ and $L_{2}$ are
symplectic. We shall use $\mu $ instead of $\lambda $ as a parameter, and $%
( \cdot ) ^{\prime }\equiv \frac{\partial }{\partial \mu }(
\cdot ) $. Then 
\begin{equation*}
S_{0}^{\prime }=
\begin{bmatrix}
A_{4}-A_{3} & B_{4}-B_{3} \\ 
B_{4}^{\ast }-B_{3}^{\ast } & C_{4}-C_{3}
\end{bmatrix}
\equiv
\begin{bmatrix}
A & B \\ 
B^{\ast } & C
\end{bmatrix}
\text{.}
\end{equation*}

If 
\begin{equation*}
\begin{bmatrix}
X( \tau ,\sigma ) & Z( \tau ,\sigma ) \\ 
W( \tau ,\sigma ) & Y( \tau ,\sigma )
\end{bmatrix}
\equiv
\begin{bmatrix}
X & Z \\ 
W & Y
\end{bmatrix}
=L_1\Phi _0( \tau ) \Phi _0^{-1}( \sigma ) 
\text{,}
\end{equation*}
then 
\begin{equation*}
C_2=\int_0^\tau \bigl( XC( \sigma ) X^{*}+ZA( \sigma )
Z^{*}-XB( \sigma ) Z^{*}-ZB^{*}( \sigma ) X^{*}\bigr)
\,d\sigma \text{.}
\end{equation*}

Hence, if $JS_{0}^{\prime }J\leq 0$, $\varphi ( \tau ,\mu _{1})
\leq \varphi ( \tau ,\mu _{2}) $ for $\mu _{1}\leq \mu _{2}$ and
we have proved the following theorem:

\begin{theorem}
If $JS_0^{\prime }J\leq 0$ ($JS_0^{\prime }J\geq 0$), $\varphi \left( \tau
,\mu \right) $ is an increasing (decreasing) function of $\mu $ for every $%
\tau $. Moreover, if, for every $\tau $, there exists $\sigma <\tau $ such
that $\left( JS_0^{\prime }J\right) \left( \sigma \right) <0$ ($JS_0^{\prime
}J>0$), then $\varphi \left( \tau ,\mu \right) $ is a strictly increasing
(decreasing)\ function of $\mu $ for every $\tau >0$.
\end{theorem}

Notice that if $L_{1}$ and $L_{2}$ are antisymplectic one has to reverse the
inequalities involving $JS_{0}^{\prime }J$\ in this theorem.

\section{Second remarkable case}

\noindent Let us take 
\begin{equation*}
\Phi =L_{1}\Phi _{0}^{\ast }L_{2}\text{,}
\end{equation*}
where 
\begin{equation*}
\dot{\Phi}_{0}=M_{0}\Phi _{0}\text{, \ }M_{0}=-JS_{0}\text{.}
\end{equation*}

$L_{1}$ and $L_{2}$ are both symplectic or both antisymplectic and
eigenvalue dependent: $L_{1}\equiv L_{1}\left( \lambda \right) $, $%
L_{2}\equiv L_{2}\left( \lambda \right) $. As before, $\Phi $, $\Phi _{0}$, $%
M_{0}$ and $S_{0}$ are both time and eigenvalue dependent: $\Phi \equiv \Phi
\left( \tau \right) \equiv \Phi \left( \tau ,\lambda \right) $, $\Phi
_{0}\equiv \Phi _{0}\left( \tau \right) \equiv \Phi _{0}\left( \tau ,\lambda
\right) $, $M_{0}\equiv M_{0}\left( \tau \right) \equiv M_{0}\left( \tau
,\lambda \right) $, $S_{0}\equiv S_{0}\left( \tau \right) \equiv S_{0}\left(
\tau ,\lambda \right) $ and so on.
\begin{equation*}
\begin{split}
M_{1} &=\dot{\Phi}\Phi ^{-1}=L_{1}\Phi _{0}^{\ast }M_{0}^{\ast }\Phi
_{0}^{\ast -1}L_{1}^{-1} \\
&=\Phi L_{2}^{-1}M_{0}^{\ast }L_{2}\Phi ^{-1}\text{.}\\
M_{2}&=\Phi ^{\prime }\Phi ^{-1}=\left( L_{1}^{\prime }\Phi _{0}^{\ast
}L_{2}+L_{1}\Phi _{0}^{\ast \prime }L_{2}+L_{1}\Phi _{0}^{\ast
}L_{2}^{\prime }\right) L_{2}^{-1}\Phi _{0}^{\ast -1}L_{1}^{-1}\text{.}\\
M_{2}&=\Phi ^{\prime }\Phi ^{-1}=L_{1}^{\prime }L_{1}^{-1}+L_{1}\Phi
_{0}^{\ast \prime }\Phi _{0}^{\ast -1}L_{1}^{-1}+L_{1}\Phi _{0}^{\ast
}L_{2}^{\prime }L_{2}^{-1}\Phi _{0}^{\ast -1}L_{1}^{-1}\text{.}
\end{split}
\end{equation*}

Notice that $\left( \Phi _{0}^{\ast \prime }\Phi _{0}^{\ast -1}\right)
^{\ast }=\Phi _{0}^{-1}\Phi _{0}^{\prime }$ is $K\equiv K\left( \tau
,\lambda \right) $, as defined in this section when we replace $\Phi $ by $%
\Phi _{0}$. In this situation, $K_{0}=0$ and $M_{1}$ is $M_{0}$.
\begin{equation*}
K\equiv K\left( \tau \right) =\int_{0}^{\tau _{{}}}\Phi _{0}^{-1}\left(
\sigma \right) M_{0}^{\prime }\left( \sigma \right) \Phi _{0}\left( \sigma
\right) d\sigma \text{.}
\end{equation*}
Then 
\begin{equation*}
\begin{split}
M_{2}\left( \tau \right) &= L_{1}^{\prime }L_{1}^{-1}+L_{1}\Phi _{0}^{\ast
}L_{2}^{\prime }L_{2}^{-1}\Phi _{0}^{\ast -1}L_{1}^{-1} \\
&\quad +L_{1}\left( \int_{0}^{\tau _{{}}}\Phi _{0}^{\ast }\left( \sigma \right)
M_{0}^{\ast \prime }\left( \sigma \right) \Phi _{0}^{\ast -1}\left( \sigma
\right) d\sigma \right) L_{1}^{-1}\text{.}\\[2pt]
M_{2}\left( \tau \right) &=L_{1}^{\prime }L_{1}^{-1}+\Phi
L_{2}^{-1}L_{2}^{\prime }\Phi ^{-1} \\
&+ L_{1}\left( \int_{0}^{\tau _{{}}}\Phi _{0}^{\ast }\left( \sigma \right)
M_{0}^{\ast \prime }\left( \sigma \right) \Phi _{0}^{\ast -1}\left( \sigma
\right) d\sigma \right) L_{1}^{-1}\text{.}
\end{split}
\end{equation*}

\begin{theorem}
\label{T6}Let $( L_2) _{22}=0$, $\det \bigl( ( L_2)_{12}\bigr)
\neq 0$, $Q_2( \tau ) =r( \tau ) \cos
\varphi ( \tau ) $ \ and \ $Q_1( \tau ) =r( \tau
) \sin \varphi ( \tau ) $. Denote $\varphi _1( \tau
) ,\ldots ,\varphi _n( \tau ) $ the eigenvalues of $\varphi
( \tau ) $. Then, if $C_0>0$ ($C_0<0$) and $\sin \varphi _j(
\tau _0) =0$, then $\varphi _j( \tau ) $ is decreasing
(increasing) in a neighborhood of $\tau _0$.
\end{theorem}

\begin{proof}
Denote 
\begin{equation*}
\begin{split}
C_3 &= -( L_2)_{12}^{*}C_0( L_2)_{12}, \\
B_3 &= -( L_2)_{12}^{*}C_0( L_2)_{11}+(L_2)_{12}^{*}B_0( L_2)_{21}, \\
A_3 &= -( L_2)_{11}^{*}C_0( L_2)_{11}-(L_2)_{21}^{*}A_0( L_2)_{21} \\
&\quad +( L_2)_{11}^{*}B_0( L_2)_{21}+( L_2)_{21}^{*}B_0^{*}( L_2)_{11}.
\end{split}
\end{equation*}
Then 
\begin{equation*}
C_1=Q_2C_3Q_2^{*}-Q_2B_3Q_1^{*}-Q_1B_3^{*}Q_2^{*}+Q_1A_3Q_1^{*}.
\end{equation*}

Let $U\equiv U( \tau ) $ a $C^1$ orthogonal matrix defined in a
neighborhood of $\tau _0$ and $\Phi =U^{*}\varphi U$. Then, as, for $k\geq 1$, 
\begin{equation*}
\mathcal{C}_\varphi ^k\dot{\varphi}=
U\bigl( -\mathcal{C}_\Phi ^{k+1}(
U^{*}\dot{U}) +\mathcal{C}_\Phi ^k\dot{\Phi}\bigr) U^{*}\text{,}
\end{equation*}
from formula \eqref{k}, one has 
\begin{equation*}
\frac{\sin \mathcal{C}_\Phi }{\mathcal{C}_\Phi }\dot{\Phi}-( \sin 
\mathcal{C}_\Phi ) ( U^{*}\dot{U}) =U^{*}r^{-1}C_1r^{*-1}U\text{.}
\end{equation*}

One can choose $U$ such that $\Phi ( \tau _0) $ is diagonal and $%
\Phi =\diag( \Phi _1,\linebreak[0]\Phi _2) $, with $\sin \Phi _1(
\tau _0) \neq 0$, $\sin \Phi _2( \tau _0) =0$.

Then, one obtains: 
\begin{equation*}
\left( \frac{\sin \mathcal{C}_\Phi }{\mathcal{C}_\Phi }\dot{\Phi}\right)_{22}
=\dot{\Phi}_2\text{,\quad}\bigl(( \sin \mathcal{C}_\Phi )
( U^{*}\dot{U}) \bigr)_{22}\left( \tau _0\right) =0\text{,}
\end{equation*}
and 
\begin{equation*}
\begin{split}
U^{*}r^{-1}C_1r^{*-1}U &=\cos \Phi UC_3U^{*}\cos \Phi -\cos \Phi
UB_3U^{*}\sin \Phi \\
&\quad -\sin \Phi UB_3^{*}U^{*}\cos \Phi +\sin \Phi UA_3U^{*}\sin \Phi \text{.}
\end{split}
\end{equation*}
Hence 
\begin{equation*}
\left( U^{*}r^{-1}C_1r^{*-1}U\right) _{22}\left( \tau _0\right) =\left( \cos
\Phi UC_3U^{*}\cos \Phi \right) _{22}\left( \tau _0\right) <0.
\end{equation*}
and
\begin{equation*}
\dot{\Phi}_2\left( \tau _0\right) =\left( \cos \Phi UC_3U^{*}\cos \Phi
\right) _{22}\left( \tau _0\right) <0.
\end{equation*}

Then $\dot{\Phi}_2\left( \tau \right) <0$ in a neighborhood of $\tau _0$ and
the theorem follows.
\end{proof}

Similarly one can prove the following theorem:

\begin{theorem}
\label{T7}Let $( L_2) _{21}=0$, $\det ( ( L_2)
_{11}) \neq 0$, $Q_2( \tau ) =r( \tau ) \cos
\varphi ( \tau ) $ \ and \ $Q_1( \tau ) =r( \tau
) \sin \varphi ( \tau ) $. Denote $\varphi _1( \tau
) ,\ldots ,\varphi _n( \tau ) $ the eigenvalues of $\varphi
( \tau ) $. Then, if $C_0>0$ ($C_0<0$) and $\cos \varphi _j(
\tau _0) =0$, then $\varphi _j( \tau ) $ is decreasing
(increasing) in a neighborhood of $\tau _0$.
\end{theorem}

\subsection{Example}
\ \\
\noindent
Let $A_{0}=( 1-\mu ) A_{3}+\mu A_{4}$, $B_{0}=( 1-\mu
) B_{3}+\mu B_{4}$, $C_{0}=( 1-\mu ) C_{3}+\mu C_{4}$.
Assume that $A_{3}$, $A_{4}$, $B_{3}$, $B_{4}$, $C_{3}$ and $C_{4}$ are $\mu 
$-independent. We shall use $\mu $ instead of $\lambda $ as a parameter, and 
$( \cdot ) ^{\prime }\equiv \frac{\partial }{\partial \mu }(
\cdot ) $. Then 
\begin{equation*}
S_{0}^{\prime }=
\begin{bmatrix}
A_{4}-A_{3} & B_{4}-B_{3} \\ 
B_{4}^{\ast }-B_{3}^{\ast } & C_{4}-C_{3}
\end{bmatrix}
\equiv
\begin{bmatrix}
A & B \\ 
B^{\ast } & C
\end{bmatrix}
\text{.}
\end{equation*}

Define
\begin{equation*}
L_{1}=
\begin{bmatrix}
\alpha _{0} & -\beta _{0} \\ 
\beta _{0} & \alpha _{0}
\end{bmatrix}
\text{,\quad} L_{2}=
\begin{bmatrix}
( 1-\mu ) \delta _{3}+\mu \delta _{4} & -I_{n} \\ 
I_{n} & 0
\end{bmatrix}
\end{equation*}
with $( \alpha _{0}\alpha _{0}^{\ast }+\beta _{0}\beta _{0}^{\ast
}) ^{-1/2}=I_{n}$, $\alpha _{0}\beta _{0}^{\ast }=\beta _{0}\alpha
_{0}^{\ast }$, $\delta _{3}=\delta _{3}^{\ast }$ and $\delta _{4}=\delta
_{4}^{\ast }$.

If 
\begin{equation*}
\begin{bmatrix}
X( \tau ) & Z( \tau ) \\ 
W( \tau ) & Y( \tau )
\end{bmatrix}
\equiv
\begin{bmatrix}
X & Z \\ 
W & Y
\end{bmatrix}
= L_{1}\Phi _{0}^{\ast }( \tau ) \text{,}
\end{equation*}
then
\begin{multline}\label{p}
C_{2} = Q_{1}( \delta _{4}-\delta _{3}) Q_{1}^{\ast }   \\
\quad -\int_{0}^{\tau }\!\!( ZC( \sigma ) Z^{\ast }+XA( \sigma
) X^{\ast }+XB^{\ast }( \sigma ) Z^{\ast }+ZB( \sigma
) X^{\ast }) \,d\sigma \text{.}\!\!\!\!
\end{multline}

Hence, if $S_{0}^{\prime }\leq 0$ and $\delta _{4}-\delta _{3}\geq 0$, $%
\varphi ( \tau ,\mu _{1}) \leq \varphi ( \tau ,\mu
_{2}) $ for $\mu _{1}\leq \mu _{2}$ and we have proved the following
theorem:

\begin{theorem}
\label{T8}If $S_0^{\prime }\leq 0$ and $\delta _4-\delta _3\geq 0$, $\varphi
( \tau ,\mu ) $ is an increasing function of $\mu $ for every $%
\tau $. Moreover, if $\delta _4-\delta _3>0$ or, for every $\tau $, there
exists $\sigma <\tau $ such that $( S_0^{\prime }) ( \sigma
) <0$, then $\varphi ( \tau ,\mu ) $ is a strictly
increasing function of $\mu $ for every $\tau >0$.
\end{theorem}

\subsection{Example: the Sturm-Liouville problem}
\ \\
\noindent
Consider the Sturm-Liouville equation
\begin{equation}
\left( C_{0}^{-1}\dot{q}\right) ^{\cdot }+\left( -D+\lambda E\right) q=0%
\text{,}  \label{r}
\end{equation}
subject to the separated end conditions
\begin{equation}\label{s}
\begin{aligned}
\beta _{0}q\left( 0\right) +\alpha _{0}\left( C_{0}^{-1}\dot{q}\right)
\left( 0\right) &=0\\
\delta _{1}q\left( t\right) +\gamma _{1}\left( C_{0}^{-1}\dot{q}\right)
\left( t\right) &=0.
\end{aligned}
\end{equation}

In this case $A_{0}=-D+\lambda E$, $B_{0}=0$; $C_{0}$, $D$ and $E$ are $\tau 
$ dependent and $\lambda $ independent; $C_{0},E>0$. The matrices $\alpha
_{0}$, $\beta _{0}$, $\gamma _{1}$, $\delta _{1}$ are $\lambda $
independent. In this case $\beta _{1}=\alpha _{1}=\delta _{0}=\gamma _{0}=0$%
One also has 
\begin{equation*}
\alpha _{0}\beta _{0}^{\ast }=\beta _{0}\alpha _{0}^{\ast }\text{,\quad}%
\gamma _{1}\delta _{1}^{\ast }=\delta _{1}\gamma _{1}^{\ast }\text{.}
\end{equation*}

Assume also that $\alpha _{0}\alpha _{0}^{\ast }+\beta _{0}\beta _{0}^{\ast
}>0$, $\det \gamma _{1}\neq 0$. It is clear that one can replace $\delta
_{1} $ by $\gamma _{1}^{-1}\delta _{1}\equiv \delta $ (a symmetric matrix)
and$\ \gamma _{1}$ by $I_{n}$. One can also replace $\alpha _{0}$ by $\left(
\alpha _{0}\alpha _{0}^{\ast }+\beta _{0}\beta _{0}^{\ast }\right)
^{-1/2}\alpha _{0}$ and $\beta _{0}$ by $\left( \alpha _{0}\alpha _{0}^{\ast
}+\beta _{0}\beta _{0}^{\ast }\right) ^{-1/2}\beta _{0}$ and have $\alpha
_{0}\alpha _{0}^{\ast }+\beta _{0}\beta _{0}^{\ast }=I_{n}$, as we shall
assume from now on. Then condition \eqref{a} is 
\begin{equation*}
\det \left(
\begin{bmatrix}
\beta _{0} & \alpha _{0} \\ 
0 & 0
\end{bmatrix}
-
\begin{bmatrix}
0 & 0 \\ 
\delta & I_{n}
\end{bmatrix}
\begin{bmatrix}
Q_{c}\left( t\right) & Q_{s}\left( t\right) \\ 
P_{c}\left( t\right) & P_{s}\left( t\right)
\end{bmatrix}
\right) =0\text{.}
\end{equation*}

Defining 
\begin{align*}
Q_{2} &= \left( \alpha _{0}Q_{c}^{\ast }-\beta _{0}Q_{s}^{\ast }\right)
\delta +\alpha _{0}P_{c}^{\ast }-\beta _{0}P_{s}^{\ast }\text{,} \\
Q_{1} &= -\alpha _{0}Q_{c}^{\ast }+\beta _{0}Q_{s}^{\ast }\text{,}
\end{align*}
one has that condition \eqref{a} is $\det Q_{2}\left( t\right) =0$.

From now on we shall use the notation 
\begin{equation*}
Q_{1}=r( \tau ,\lambda ) \sin \varphi ( \tau ,\lambda
) \text{,\quad}Q_{2}=r( \tau ,\lambda ) \cos \varphi
( \tau ,\lambda ) \text{.}
\end{equation*}

Notice that the continuity condition on $\varphi \left( \tau ,\lambda
\right) $ implies that $\lambda \mapsto \varphi \left( 0,\lambda \right) $
is constant.

We define $\Phi =L_{1}\Phi _{0}^{\ast }L_{2}$, $\Phi $ as in formula
\eqref{n} and
\begin{equation*}
L_{1}=
\begin{bmatrix}
\alpha _{0} & -\beta _{0} \\ 
\beta _{0} & \alpha _{0}
\end{bmatrix}
\text{,\quad}
L_{2}=
\begin{bmatrix}
\delta & -I_{n} \\ 
I_{n} & 0
\end{bmatrix}
\text{.}
\end{equation*}
Then, if 
\begin{equation*}
\begin{bmatrix}
X( \tau ) & Z( \tau ) \\ 
W( \tau ) & Y( \tau )
\end{bmatrix}
\equiv
\begin{bmatrix}
X & Z \\ 
W & Y
\end{bmatrix}
 =L_{1}\Phi _{0}^{\ast }( \tau ) \text{,}
\end{equation*}
we have
\begin{align*}
X &\equiv X( \tau ) =\alpha _{0}Q_{c}^{\ast }( \tau )
-\beta _{0}Q_{s}^{\ast }( \tau ) =-Q_{1} \\
Z &\equiv Z( \tau ) =\alpha _{0}P_{c}^{\ast }( \tau )
-\beta _{0}P_{s}^{\ast }( \tau )\\
C_{1}&=-ZC_{0}Z^{\ast }-XA_{0}X^{\ast }\\
M_{2}&=\int_{0}^{\tau _{{}}}
\begin{bmatrix}
X & Z \\ 
W & Y
\end{bmatrix}
\begin{bmatrix}
0 & -E \\ 
0 & 0
\end{bmatrix}
\begin{bmatrix}
Y^{\ast } & -Z^{\ast } \\ 
-W^{\ast } & X^{\ast }
\end{bmatrix}
d\sigma
\end{align*}
\begin{equation}
C_{2}=-\int_{0}^{\tau _{{}}}X( \sigma ) \, E( \sigma ) \,
X^{\ast }( \sigma ) \, d\sigma \text{.}  \label{q}
\end{equation}
We remark that $C_{2}<0$, for $\tau \in \left] 0,t\right] $.

\begin{lemma}
Consider the simpler case where $C_0=cI_n$, $D=dI_n$, $E=eI_n$, $\delta
=\theta I_n$\ with $c,d,e,\theta \in \mathbb{R}$, $c,e>0$. Then, there exists a
symmetric matrix $\varphi ^{-}$\ such that, for every $\tau \in \left]
0,t\right] $, 
\begin{equation*}
\lim_{\lambda \rightarrow +\infty }\varphi ( \tau ,\lambda )
=-\infty \text{,\quad }\lim_{\lambda \rightarrow -\infty }\varphi (
\tau ,\lambda ) =\varphi ^{-}\text{,}
\end{equation*}
where $\tan \varphi ^{-}=0$. Moreover, $\varphi ^{-}$ is constant for $\tau
\in \left] 0,t\right] $.
\end{lemma}

\begin{proof}
Consider first $\lambda >d/e$. Define $\omega =\sqrt{c\left( -d+\lambda
e\right) }$. Then 
\begin{align*}
Q_2 &=\theta \left( \left( \cos \omega \tau \right) \alpha _0-c\omega
^{-1}\left( \sin \omega \tau \right) \beta _0\right) \\
&\quad-\left( \cos \omega \tau \right) \beta _0-c^{-1}\omega \left( \sin \omega
\tau \right) \alpha _0\text{,} \\
Q_1 &=-\left( \cos \omega \tau \right) \alpha _0+c\omega ^{-1}\left( \sin
\omega \tau \right) \beta _0\text{.}
\end{align*}

Defining $\psi $ and $\rho $, $\det \rho \neq 0$, such that 
\begin{equation*}
\alpha _0=\rho \cos \psi \text{,\quad}c\omega ^{-1}\beta _0=\rho \sin \psi 
\text{,}
\end{equation*}
one has 
\begin{align*}
Q_2 &=\rho \bigl( \theta \cos ( \omega \tau I_n+\psi )
-c^{-1}\omega \sin ( \omega \tau I_n+\psi ) \bigr) \text{,} \\
Q_1 &=-\rho \cos ( \omega \tau I_n+\psi ) \text{.}
\end{align*}

Then $Q_1^{-1}Q_2=-\theta +c\omega ^{-1}\tan ( \omega \tau I_n+\psi
) $, for every $\tau $ such that $\det \cos ( \omega \tau
I_n+\psi ) \neq 0$.

Hence 
\begin{align*}
Q_1 &=\rho \tilde{\rho}\sin \zeta ( -\theta ,c\omega ^{-1},\omega \tau
I_n+\psi ) \text{,} \\
Q_2 &=\rho \tilde{\rho}\cos \zeta ( -\theta ,c\omega ^{-1},\omega \tau
I_n+\psi ) \text{,}
\end{align*}
with $\zeta $ defined by \eqref{m} and 
\begin{equation*}
\tilde{\rho}=\sqrt{\cos ^2( \omega \tau I_n+\psi ) + \bigl( \theta
\cos ( \omega \tau I_n+\psi ) -c^{-1}\omega \sin ( \omega
\tau I_n+\psi ) \bigr) ^2}\text{.}
\end{equation*}

As $Q_1=r\sin \varphi $, $Q_2=r\cos \varphi $, one has 
\begin{equation*}
r=\rho \tilde{\rho}\text{,\quad}\varphi =\zeta ( -\theta ,c\omega
^{-1},\omega \tau I_n+\psi ) \text{.}
\end{equation*}
As 
\begin{equation*}
\lim_{\sigma \rightarrow +\infty }\zeta ( -\theta ,c\omega ^{-1},\sigma
) =-\infty \text{,}
\end{equation*}
the first part of the lemma follows.

Consider now the case $\lambda <d/e$. Define $\omega =\sqrt{c(
d-\lambda e) }$. Then 
\begin{align*}
Q_2 &=\theta \bigl( ( \cosh \omega \tau )\, \alpha _0-c\omega
^{-1}( \sinh \omega \tau )\, \beta _0\bigr) \\
&\quad -( \cosh \omega \tau )\, \beta _0+c^{-1}\omega ( \sinh
\omega \tau )\, \alpha _0\text{,} \\
Q_1 &=-( \cosh \omega \tau )\, \alpha _0+c\omega ^{-1}( \sinh
\omega \tau )\, \beta _0\text{.}
\end{align*}
Defining $\eta $ and $\varrho $, $\det \varrho \neq 0$, such that 
\begin{equation*}
\alpha _0=\varrho \cos \eta \text{,\quad}\beta _0=\varrho \sin \eta \text{,}
\end{equation*}
Then 
\begin{equation*}
Q_2^{-1}Q_1=\frac{-\cos \eta +c\omega ^{-1}( \tanh \omega \tau )
\sin \eta }{\bigl( \theta +c^{-1}\omega ( \tanh \omega \tau )\bigr) 
\cos \eta -\bigl( \theta c\omega ^{-1}( \tanh \omega \tau
) +1\bigr) \sin \eta }
\end{equation*}
Hence, for every $\tau \in \left] 0,t\right] $, there exists a $\lambda _{*}$
such that, for $\lambda \leq \lambda _{*}$, 
\begin{equation*}
\left\| Q_2^{-1}Q_1\right\| \leq \left( -\left| \theta \right| +c^{-1}\omega
\left( \tanh \omega \tau \right) \right) ^{-1}\text{,}
\end{equation*}
and 
\begin{equation*}
\lim_{\lambda \rightarrow -\infty }\left\| Q_2^{-1}Q_1\right\| =0
\end{equation*}

For $\tau _{*}>0$, this convergence is uniform in $\left[ \tau _{*},t\right] 
$. From this, the last part of the lemma follows.
\end{proof}

\begin{theorem}
\label{T9}Consider the general case for $C_0$, $D$, $E$ and $\delta $. Then,
for every $\tau \in \left] 0,t\right] $, 
\begin{equation*}
\lim_{\lambda \rightarrow +\infty }\varphi \left( \tau ,\lambda \right)
=-\infty \text{,\quad}\lim_{\lambda \rightarrow -\infty }\tan \varphi
\left( \tau ,\lambda \right) =0\text{,}
\end{equation*}
and $\varphi \left( \tau ,\lambda \right) $ is a strictly decreasing
function of $\lambda$.

Moreover, the eigenvalues of $\varphi \left( \tau ,\lambda \right) $
converge to constant functions on $\left] 0,t\right] $, as $\lambda
\rightarrow -\infty $.
\end{theorem}

\begin{proof}
As $C_2$, defined by formula \eqref{q}, is $<0$, $\varphi \left( \tau
,\lambda \right) $ is a strictly decreasing function of $\lambda $, for
every $\tau \in \left] 0,t\right] $.

For $\lambda >0$, choose $\theta >\left\| \delta \right\| $, $d\geq D$, $%
0<e\leq E$, $0<c\leq C_0$, with $\theta ,d,$ $e,c\in \mathbb{R}$.

We use now Theorem \ref{T8}. Put $\delta _3=\delta $, $\delta _4=\theta I_n$%
, $A_3=-D+\lambda E$, $A_4=\left( -d+\lambda e\right) I_n$, $C_3=C_0$, $%
C_4=cI_n$.

Then, from Theorem \ref{T8}, one concludes that 
\begin{equation*}
\varphi \left( \tau ,\lambda \right) \equiv \varphi \left( \tau ,\lambda
,0\right) <\varphi \left( \tau ,\lambda ,1\right) \text{,}
\end{equation*}
and the first formula of the theorem is proved.

For $\lambda <0$, choose $\theta >\left\| \delta \right\| $, $d\geq D$, $%
e\geq E$, $0<c\leq C_0$, with $\theta ,d,$ $e,c\in \mathbb{R}$.

We use again Theorem \ref{T8}. Put $\delta _3=\delta $, $\delta _4=\theta
I_n $, $A_3=-D+\lambda E$, $A_4=\left( -d+\lambda e\right) I_n$, $C_3=C_0$, $%
C_4=cI_n$.

Then, from Theorem \ref{T8}, one concludes that 
\begin{equation*}
\varphi _1\left( \tau ,\lambda ,0\right) \equiv \varphi \left( \tau ,\lambda
\right) \equiv \varphi \left( \tau ,\lambda ,0\right) <\varphi \left( \tau
,\lambda ,1\right) \equiv \varphi _1\left( \tau ,\lambda ,1\right) \text{,}
\end{equation*}
the eigenvalues of $\varphi \left( \tau ,\lambda \right) $ are bounded as $%
\lambda \rightarrow -\infty $.

For $\lambda <0$, choose $\theta >\left\| \delta \right\| $, $d\leq D$, $%
0<e\leq E$, $c\geq C_0$, with $\theta ,d,$ $e,c\in \mathbb{R}$.

We use once more Theorem \ref{T8}. Put $\delta _3=\delta $, $\delta
_4=-\theta I_n$, $A_3=-D+\lambda E$, $A_4=( -d+\lambda e) I_n$, $%
C_3=C_0$, $C_4=cI_n$.

Then, from Theorem \ref{T8}, one concludes that 
\begin{equation*}
\varphi _2( \tau ,\lambda ,0) \equiv \varphi ( \tau ,\lambda
) \equiv \varphi ( \tau ,\lambda ,0) >\varphi ( \tau
,\lambda ,1) \equiv \varphi _2( \tau ,\lambda ,1) \text{.}
\end{equation*}

Choose $\lambda _{*}$ the minimum of the $\lambda <0$ such that $\det \cos
\varphi _1( \tau ,\lambda ,\mu ) \linebreak[0]=0$ or $\det \cos \varphi
_1( \tau ,\lambda ,\mu ) =0$, with $\mu \in \left[ 0,1\right] $.
It is clear that there exists such a $\lambda _{*}$, as $\varphi _1$ and $%
\varphi _2$ are bounded near $\lambda =-\infty $. Then, for $\lambda
<\lambda _{*}$ and $\mu \in \left[ 0,1\right] $, $\det \cos \varphi _1(
\tau ,\lambda ,\mu ) \neq 0$, $\det \cos \varphi _2( \tau
,\lambda ,\mu ) \neq 0$. Hence, $\det Q_2( \tau ,\lambda ,\mu
) \neq 0$ in both cases.

As, from \eqref{o} and \eqref{p}, $\frac d{d\mu }Q_2^{-1}Q_1>0$ in the first
case and $<0$ in the second one, one obtains that, for $\lambda 
<\lambda_{*} $, 
\begin{equation*}
\tan \varphi_2( \tau ,\lambda ,1) <Q_2^{-1}Q_1<
\tan \varphi_1( \tau ,\lambda ,1) \text{.}
\end{equation*}
Therefore 
\begin{equation*}
\left\| Q_2^{-1}Q_1\right\| <\max \left\{ \left\| \tan \varphi _1( \tau
,\lambda ,1) \right\| ,\left\| \tan \varphi _2( \tau ,\lambda
,1) \right\| \right\} \text{.}
\end{equation*}
From Theorem \ref{T8}, one concludes that
\begin{equation*}
\lim_{\lambda \rightarrow -\infty }\left\| Q_2^{-1}Q_1\right\| =0\text{.}
\end{equation*}

Then, for $\tau >0$, 
\begin{equation*}
\lim_{\lambda \rightarrow -\infty }\tan \varphi
_1( \tau ,\lambda ,\mu ) =0
\text{ \ and }
\lim_{\lambda \rightarrow
-\infty }\tan \varphi _2( \tau ,\lambda ,\mu ) =0.
\end{equation*}
As $\lim_{\lambda \rightarrow -\infty }\varphi _1( \tau ,\lambda ,1) $
and $\lim_{\lambda \rightarrow -\infty }\varphi _2( \tau ,\lambda
,1) $\ are constant in $\left] 0,t\right] $, and the eigenvalues of
these limit functions are integer multiple of $\pi $, the continuity of the
functions $\varphi _1$\ and $\varphi _2$\ implies the last part of the
theorem.
\end{proof}

Finally we have the following theorem:

\begin{theorem}
For the Sturm-Liouville equation \eqref{r}, subject to conditions \eqref{s},
there are an infinite number of eigenvalues $\lambda _{j,0}<\lambda
_{j,1}<\lambda _{j,2}<\cdots <\lambda _{j,k}<\cdots $, $j=1,2,\ldots ,n$,
with $\lim_{k\rightarrow \infty }\lambda _{j,k}=+\infty $.

The eigenfunctions can be described as follows. There exists a matrix
function $Q_1( \tau ,\lambda ) =r( \tau ,\lambda )
\sin \varphi ( \tau ,\lambda ) $, such that $\det r( \tau
,\lambda ) \linebreak[0]\neq 0$ and $\varphi ( \tau ,\lambda ) $ is
symmetric. The matrix functions $r$ and $\varphi $ are continuous. Consider
the $\varphi $\ eigenvalues $\varphi _j( \tau ,\lambda ) $ and
eigenvectors $e_j( t,\lambda _{j,k}) $. Then the eigenfunction
corresponding to $\lambda _{j,k}$ is $Q_1( \tau ,\lambda _{j,k})\linebreak[0]
e_j( t,\lambda _{j,k}) $ and $\sin \varphi _j( \tau ,\lambda
_{j,k}) $ has exactly $k$ zeros on $\left] 0,t\right[ $.
\end{theorem}

\begin{proof}
Consider $\varphi ( \tau ,\lambda ) $ and its eigenvalues $%
\varphi _j( \tau ,\lambda ) $, $j=1,2,\linebreak[0]\ldots ,\linebreak[0]n$.
Then, from
Theorem \ref{T9}, $\varphi _j( \tau ,\lambda ) $ is strictly
decreasing in $\lambda $, $\lim_{\lambda \rightarrow +\infty }\varphi
_j( \tau ,\lambda ) =-\infty $, and there exists $l_j\in \mathbb{Z}$%
, such that $\lim_{\lambda \rightarrow -\infty }\varphi _j( \tau
,\lambda ) =l_j\pi $, for $\tau \in \left] 0,t\right] $.

From Theorem \ref{T6}, whenever $\varphi _j( \tau _l,\lambda )
=l\pi $, for some $\tau _l\in \left] 0,t\right[ $, then $\varphi _j(
\tau ,\lambda ) $ is a decreasing function of $\tau $ in a
neighborhood of $\tau _l$. Then, $\varphi _j( \tau ,\lambda )
<l\pi $ for $\tau >\tau _l$\ and $\varphi _j( \tau ,\lambda )
>l\pi $ for $\tau <\tau _l$.

Clearly there exists a $\lambda _{j,k}$ such that $\varphi _j(
t,\lambda _{j,k}) =( l_j-k-\frac 12)\,\pi $, for $%
k=0,1,2,\ldots $.

For $\tau _{*}>0$, there exists $\lambda _{*}$ such that $\varphi _j(
\tau _{*},\lambda _{*}) =( l_j-1)\, \pi $. Hence, for $\tau
<\tau _{*}$, $\varphi _j( \tau _{*},\lambda _{*}) >(
l_j-1) \,\pi $. Therefore $\varphi _j( 0,\lambda _{*})
>( l_j-1)\, \pi $. As $\lambda \mapsto \varphi _j( 0,\lambda
) $ is constant, it follows that $\varphi _j( 0,\lambda )
>( l_j-1)\, \pi $ for every $\lambda $.

Define $\tau _m$, $m=1,2,\ldots ,k$, $\varphi _j( \tau _m,\lambda
_{j,k}) =( l_j-m) \,\pi $. The points $\tau _m$\ are the
unique points where $\sin \varphi _j( \tau ,\lambda _{j,k}) =0$
for $\tau \in \left] 0,t\right] $.
\end{proof}

\section*{Acknowledgements}

\noindent The Mathematical Physics Group is supported by the
portuguese Foundation for Science and Technology (FCT)

\appendix
\section{}

\begin{proposition}
Let $n=1$. $L_0+L_1\Phi L_2$ is symplectic for every symplectic matrix $\Phi 
$ is equivalent to $\left( \det L_0\right) +\left( \det L_1\right) \left(
\det L_2\right)\linebreak[0]=1$ and $L_1^{*}JL_0JL_2^{*}=0$.

If $L_0+L_1\Phi L_2$ is symplectic for every symplectic matrix $\Phi $, one
of the following situations happens

a) $L_0$ is symplectic and $\det L_1=\det L_2=0$, with $L_1\neq 0$ and $%
L_2\neq 0$.

b) $L_0$ is symplectic and $L_1=0$ or $L_2=0$.

c) $L_0=0$ and $\det L_1\det L_2=1$.
\end{proposition}

\begin{proof}
\begin{equation*}
\begin{split}
&( L_0+L_1\Phi L_2) J( L_0^{*}+L_2^{*}\Phi ^{*}L_1^{*}) \\
&= L_0JL_0^{*}+L_0JL_2^{*}\Phi ^{*}L_1^{*}+L_1\Phi L_2JL_0^{*}+L_1\Phi
L_2JL_2^{*}\Phi ^{*}L_1^{*} \\
&= ( \det L_0) J+L_0JL_2^{*}\Phi ^{*}L_1^{*}+L_1\Phi
L_2JL_0^{*}+( \det L_1) ( \det L_2) J = J.
\end{split}
\end{equation*}

As this must be true for $\Phi $ and $-\Phi $, one has 
\begin{align*}
( \det L_0) +( \det L_1) ( \det L_2) &=1,\\
L_0JL_2^{*}\Phi ^{*}L_1^{*}+L_1\Phi L_2JL_0^{*} &= 0.
\end{align*}

Hence, $L_1\Phi L_2JL_0^{*}$ is symmetric, for every symplectic matrix $\Phi 
$. As $L_1\left( \Phi _1+\Phi _2\right) L_2JL_0^{*}$ is also symmetric for
any two symplectic matrices, $L_1\Phi L_2JL_0^{*}$ is symmetric even if $%
\Phi $ is not symplectic. As $K_1\Phi K_2$ is symmetric for every matrix $%
\Phi $ if and only if $K_2JK_1=0$, one easily concludes that $%
L_1^{*}JL_0JL_2^{*}=0$. The proposition follows now without problems.
\end{proof}

Let $n=1$ and $f_{11},f_{12},f_{21},f_{22}:\mathbb{R}^{4}\rightarrow \mathbb{R}$
four affine functions. Then, if 
\begin{equation*}
L=
\begin{bmatrix}
f_{11}( \Phi _{11},\Phi _{12},\Phi _{21},\Phi _{22}) & 
f_{12}( \Phi _{11},\Phi _{12},\Phi _{21},\Phi _{22}) \\ 
f_{21}( \Phi _{11},\Phi _{12},\Phi _{21},\Phi _{22}) & 
f_{22}( \Phi _{11},\Phi _{12},\Phi _{21},\Phi _{22})
\end{bmatrix}
\end{equation*}
is symplectic for every symplectic matrix $\Phi $, one has that $L$ is one
of the forms 
\begin{equation*}
L=L_{0}+L_{1}\Phi L_{2}\text{,\quad}L=L_{0}+L_{1}\Phi ^{\ast }L_{2}\text{.}
\end{equation*}

This can be proved by an explicit, and tedious, computation.

Notice that, following the proposition $L_{0}$ is either $0$ or symplectic.
If $L_{0}=0$, then $L_{1}$ and $L_{2}$ can be chosen such that $\left| \det
L_{1}\right| =\left| \det L_{2}\right| =1$, $( \det L_{1}) (
\det L_{2}) =1$. In this case they are either both symplectic or both
antisymplectic.

In our problem $\Phi _{11}\equiv Q_{c}( t) =Q_{c}^{\ast }(
t) $, $\Phi _{12}\equiv Q_{s}( t) =Q_{s}^{\ast }(
t) $, $\Phi _{21}\equiv P_{c}( t) =P_{c}^{\ast }(
t) $, $\Phi _{22}\equiv P_{s}( t) =P_{s}^{\ast }(
t) $. Hence 
\begin{equation*}
f_{11}( \Phi _{11},\Phi _{12},\Phi _{21},\Phi _{22})
=x_{0}+x_{1}\Phi _{11}+x_{2}\Phi _{12}+x_{3}\Phi _{21}+x_{4}\Phi _{22}
\end{equation*}
where
\begin{align*}
x_{0} &=R( \beta _{0}\alpha _{1}-\alpha _{0}\beta _{1}+\delta
_{0}\gamma _{1}-\gamma _{0}\delta _{1}) \\
x_{1} &=R( \alpha _{0}\delta _{1}-\delta _{0}\alpha _{1}) \\
x_{2} &=R( \delta _{0}\beta _{1}-\beta _{0}\delta _{1}) \\
x_{3} &=R( \alpha _{0}\gamma _{1}-\gamma _{0}\alpha _{1}) \\
x_{4} &=R( \gamma _{0}\beta _{1}-\beta _{0}\gamma _{1})
\end{align*}
where $R=R_{0}R_{1}$ is a real eigenvalue dependent parameter, $R\neq 0$.

Notice that $x_{1}x_{4}-x_{2}x_{3}=R^{2}( \delta _{1}\gamma _{0}-\delta _{0}\gamma
_{1}) ( \beta _{1}\alpha _{0}-\alpha _{1}\beta _{0})$.
As $x_{0}=2R( \beta _{0}\alpha _{1}-\alpha _{0}\beta _{1})
=2R( \delta _{0}\gamma _{1}-\gamma _{0}\delta _{1}) $, one has
that 
\begin{equation*}
x_{1}x_{4}-x_{2}x_{3}=4^{-1}x_{0}^{2}\text{.}
\end{equation*}

Let $L_{0}=I_{2}$, the $2\times 2$ unit matrix. Then $L$ can be of the
following three forms:

a) $f_{11}=1$, $f_{22}=1$, $f_{12}=0$;

b) $f_{11}=1$, $f_{22}=1$, $f_{21}=0$;

c) there exists an $\kappa \neq 0$ such that $f_{22}-1=-(
f_{11}-1) $, $f_{12}=\kappa ( f_{11}-1) $, $f_{12}=-\kappa
^{-1}( f_{11}-1) $.

The case where $L_{0}$ is symplectic but $\neq I_{2}$ is easily derived from
this one.

Let now $L_{0}=0$. Then $x_{0}=0$ and $x_{1}x_{4}-x_{2}x_{3}=0$.

There are five possible situations: a) $x_{1}\neq 0$, b) $x_{1}=0,x_{4}\neq
0,x_{3}=0$, c) $x_{1}=0,x_{4}\neq 0,x_{2}=0$, d) $x_{1}=0,x_{4}=0,x_{2}=0$,
e) $x_{1}=0,x_{4}=0,x_{3}=0$. 
\begin{equation*}
\centering
\begin{tabular}{llp{4em}p{4em}ll}
\toprule
& a) & b) & c) & d) & e) \\
\cmidrule{2-6}
$( L_{1})_{11}$ & $a$ & $ax_{2}x_{4}^{-1}$ & $0$ & $0$ & $a$ \\[6pt]
$( L_{1})_{12}$ & $ax_{3}x_{1}^{-1}$ & $a$ & $a$ & $a$ & $0$ \\[6pt]
$( L_{1})_{21}$ & $b$ & $-\nu a^{-1}\linebreak+bx_{2}x_{4}^{-1}$ & $-\nu
a^{-1}$ & $-\nu a^{-1}$ & $b$ \\[23pt]
$( L_{1})_{22}$ & $\nu a^{-1}+bx_{3}x_{1}^{-1}$ & $b$ & $b$ & $b$
& $\nu a^{-1}$ \\[6pt]
$( L_{2})_{11}$ & $a^{-1}x_{1}$ & $0$ & $a^{-1}x_{3}$ & $%
a^{-1}x_{3}$ & $0$ \\[6pt]
$( L_{2})_{12}$ & $c$ & $-\nu ax_{4}^{-1}$ & $-\nu
ax_{4}^{-1}\linebreak+cx_{3}x_{4}^{-1}$ & $c$ & $-\nu ax_{2}^{-1}$ \\[23pt]
$( L_{2})_{21}$ & $a^{-1}x_{2}$ & $a^{-1}x_{4}$ & $a^{-1}x_{4}$
& $0$ & $a^{-1}x_{2}$ \\[6pt]
$( L_{2})_{22}$ & $\nu ax_{1}^{-1}+cx_{2}x_{1}^{-1}$ & $c$ & $c$
& $\nu ax_{3}^{-1}$ & $c$ \\
\bottomrule
\end{tabular}
\end{equation*}
where $a$, $b$ and $c$ are real eigenvalue dependent parameters, $a\neq 0$,
and $\nu =\pm 1$; $\nu =1$ in the symplectic case, $\nu =-1$ in the
antisymplectic case.

\end{document}